\documentclass{IEEEojcsys}

\usepackage[colorlinks,urlcolor=blue,linkcolor=blue,citecolor=blue]{hyperref}

\usepackage{color,array}

\usepackage{graphicx}
\graphicspath{{../}}
\usepackage{amsmath, amssymb}
\usepackage{mathrsfs}
\usepackage{graphicx}
\usepackage{cite}
\usepackage{hyperref}
\usepackage{url}
\usepackage{subcaption}
\usepackage[table]{xcolor}
\usepackage{booktabs}
\usepackage{amsfonts}
\usepackage{nicefrac}
\usepackage{microtype}
\usepackage[linesnumbered,ruled,vlined]{algorithm2e}
\usepackage{xcolor}
\usepackage{doi}
\usepackage{soul}
\usepackage{tabularx}
\usepackage{array}
\usepackage{booktabs}
\usepackage{makecell}
\usepackage[table]{xcolor}

\newcolumntype{Y}{>{\centering\arraybackslash}X}

\makeatletter

\def\@begintheorem#1#2{%
  \par
  \addvspace{6pt plus 1pt minus 1pt}%
  \noindent
  \textbf{#1~#2:}\ %
  \ignorespaces
}

\def\@opargbegintheorem#1#2#3{%
  \par
  \addvspace{6pt plus 1pt minus 1pt}%
  \noindent
  \textbf{#1~#2 (#3):}\ %
  \ignorespaces
}

\def\@endtheorem{%
  \par
  \addvspace{3pt plus 1pt minus 1pt}%
}

\def\IEEEproof{%
  \@ifnextchar[{\@IEEEproof}{\@IEEEproof[\IEEEproofname]}%
}

\def\@IEEEproof[#1]{%
  \par
  \addvspace{6pt plus 1pt minus 1pt}%
  \noindent
  \textit{#1:}\ %
  \ignorespaces
}

\def\endIEEEproof{%
  \hspace*{\fill}\IEEEQED
  \par
  \addvspace{3pt plus 1pt minus 1pt}%
}

\makeatother
\newtheorem{theorem}{Theorem}
\newtheorem{assumption}{Assumption}
\newtheorem{lemma}{Lemma}

\newtheorem{definition}{Definition}
\newtheorem{remark}{Remark}

\usepackage{tabularx, booktabs, amssymb, pifont}
\newcommand{\cmark}{\ding{51}}
\newcommand{\xmark}{\ding{55}}

\DeclareMathOperator{\Tr}{\mathrm{Tr} }

\begin{document}

\sptitle{Article Category}

\title{On the Gradient Domination of the LQG Problem}

\author{Kasra Fallah\affilmark{1}\textsuperscript{*}, 
Leonardo F. Toso\affilmark{1}\textsuperscript{*}, 
and James Anderson\affilmark{1}}

\affil{Department of Electrical Engineering, Columbia University, 
New York, NY 10027 USA}

\corresp{CORRESPONDING AUTHOR: Kasra Fallah 
(e-mail: \href{mailto:kasra.fallah@columbia.edu}
{kasra.fallah@columbia.edu})}

\authornote{\textsuperscript{*}These authors contributed equally to this work.}

\begin{abstract}
We consider solutions to the linear quadratic Gaussian (LQG) regulator problem via policy gradient (PG) methods. Although PG methods have demonstrated strong theoretical guarantees in solving the linear quadratic regulator (LQR) problem, despite its nonconvex landscape, their theoretical understanding in the LQG setting remains limited. Notably, the LQG problem lacks gradient dominance in the classical parameterization, i.e., with a dynamic controller, which hinders global convergence guarantees. In this work, we study a PG method for the LQG problem by adopting an alternative parameterization of the set of stabilizing controllers and employing a lifting argument. We refer to this parameterization as a history representation of the control input as it is parameterized by past input and output data from the previous $p$ time-steps.  Under this parameterization, the infinite-history LQG optimization reduces to an equivalent static state-feedback LQR problem, allowing us to recover gradient dominance and smoothness. For finite history length $p$, we establish model-based and model-free convergence guarantees up to a geometrically decaying history-truncation error.
\end{abstract}


\maketitle

\section{Introduction}

Policy gradient (PG) methods are popular in reinforcement learning (RL) for being model-agnostic and easy to implement without requiring direct access to a model of the underlying environment dynamics \cite{sutton1999policy}. These methods have been extensively studied in the context of linear-quadratic control design. In particular, despite its non-convex landscape, Fazel \emph{et al.} \cite{fazel_global_2019} established the global convergence of PG methods for the linear quadratic regulator (LQR) problem (referred to here as PG-LQR). The global convergence of PG-LQR is mainly due to the fact that the LQR cost is gradient dominated~\cite{polyak1963gradient} and ``almost'' smooth. This result has since inspired a broad line of work with the linear and global convergence of PG-LQR proved in \cite{mohammadi2020linear}.

PG methods have also been applied to a broader class of control problems; for a comprehensive overview, see Hu \emph{et al.}~\cite{hu_toward_2023}. For instance, Zhang \emph{et al.} \cite{zhang_policy_2020}  investigated the landscape and convergence guarantees of PG methods for $\mathcal{H}_2$ and $\mathcal{H}_\infty$ robust control settings. In addition, \cite{wang2023model, wang2023robot, toso_meta-learning_2024, zhan2025coreset, ye2024convergence} consider the multi-task collaborative LQR design under task heterogeneity. The partially observed setting is addressed in \cite{zhao_globally_2023}, where the authors propose an input-output feedback representation of the control input to obtain gradient dominance of the LQR cost under partial observability, offering a model-agnostic approach for the LQR problem with partial observation.

However, the convergence guarantees for PG methods in the presence of process and measurement noise, i.e., for the classical linear quadratic Gaussian (LQG) problem are mostly negative. In fact, existing results highlight several fundamental challenges, including the lack of gradient dominance \cite{mohammadi_lack_2021}\textcolor{black}{\cite{keivan2024case}} and the non-connectivity of the set of stabilizing controllers \cite{tang_analysis_2023,zheng2026no}.
Recent work \cite{makdah_behavioral_2022, guo_imitation_2023} has explored the history representation\footnote{Also referred to as behavioral representation in \cite{makdah_behavioral_2022, guo_imitation_2023}.} of the LQG problem, demonstrating the separation principle under such a history parameterization and leveraging imitation learning for designing a near-optimal LQG control policy.
Sadamoto and Tanaka~\cite{sadamoto_tanaka_2026} recently analyzed policy gradient directly over input-output-history policies and established convergence of policy gradient to an $\mathcal O(\epsilon)$-stationary solution. Their optimization is performed directly over the corresponding history-feedback gain. In contrast, our gradient-dominance result leverages history policies that are consistent with a state-feedback gain (this is made explicit below and is essential for the lifted LQR equivalence).

In this work, we reformulate the optimal LQG problem as a lifted LQR problem using the history representation similar to that introduced by Skelton \emph{et al.} \cite{skelton_data-based_1994} and further developed in \cite{makdah_behavioral_2022, guo_imitation_2023}. Under this reformulation, we demonstrate that an $\epsilon$-optimal controller can be found using policy gradient methods. \vspace{0.2cm}

\noindent \textbf{Contributions.} Our contributions are summarized as follows:\vspace{0.2cm}

\noindent $\bullet$ We demonstrate that, for lifted history policies satisfying a closed-loop reachability condition, policy gradient converges in model-based and model-free settings (Lemma~\ref{lemma:convergence} and Theorem~\ref{thm:K convergence}). The infinite-history lifted cost admits gradient dominance and smoothness (Lemmas~\ref{lemma:gradient dominance} and~\ref{lemma:smoothness}), while the finite history length introduces the explicit error in Lemma~\ref{lem:cost-gradient-truncation}.\vspace{0.2cm}

\noindent $\bullet$ We show that a one-point zeroth-order estimator, despite its higher variance, is sufficient for convergence, and we provide explicit performance bounds (Theorem \ref{thm:K convergence}). That accounts for the increased variance of one-point zeroth-order estimation. We note that a fundamental difficulty in LQG control compared to the noiseless 
partially observed LQR of \cite{zhao_globally_2023} is that process and measurement noise corrupt trajectory observations. This prevents the use of 
two-point gradient estimators \cite{malik_derivative-free_2020}, which require cost evaluations under the same noise realization. \vspace{0.2cm} 

\noindent $\bullet$ Convergence of PG methods requires a stabilizing initial controller,  which is nontrivial to obtain for unstable systems without model knowledge. We propose a data-driven initialization that adapts discount annealing \cite{perdomo2021stabilizing} to the LQG setting under the history representation. This provides a complete model-free approach to learn from an arbitrary initial controller to a near-optimal one. 

\begin{table*}[t]
\centering
\caption{Policy gradient methods for linear quadratic control.
\cmark\ = holds; \xmark\ = does not hold.}
\label{tab:landscape}

\footnotesize
\setlength{\tabcolsep}{3pt}
\renewcommand{\arraystretch}{1.25}

\begin{tabularx}{\textwidth}{@{}
    >{\raggedright\arraybackslash}p{0.22\textwidth}
    >{\centering\arraybackslash}p{0.065\textwidth}
    >{\centering\arraybackslash}p{0.065\textwidth}
    >{\centering\arraybackslash}p{0.065\textwidth}
    >{\centering\arraybackslash}p{0.12\textwidth}
    Y
    Y
    Y
@{}}

\toprule

\textbf{Work}
&
\makecell{\textbf{Process}\\\textbf{noise}}
&
\makecell{\textbf{Meas.}\\\textbf{noise}}
&
\makecell{\textbf{Partial}\\\textbf{obs.}}
&
\textbf{Param.}
&
\makecell{\textbf{Per-iter.}\\\textbf{stability}}
&
\makecell{\textbf{Gradient}\\\textbf{dominance}}
&
\makecell{\textbf{PG}\\\textbf{conv.}}
\\

\midrule

Fazel et al.~\cite{fazel_global_2019};
follow-ups~\cite{malik_derivative-free_2020,mohammadi2020linear}
&
\xmark
&
\xmark
&
\xmark
&
standard
&
\cmark~\cite{fazel_global_2019}
&
\cmark
&
\cmark
\\

\addlinespace[3pt]

Mohammadi et al.~\cite{mohammadi_lack_2021};
Tang et al.~\cite{tang_analysis_2023}
&
\cmark
&
\cmark
&
\cmark
&
standard
&
N/A
&
\xmark
&
\xmark
\\

\addlinespace[3pt]

Zhao et al.~\cite{zhao_globally_2023};
Xie \& Ni~\cite{xie2024data}
&
\xmark
&
\xmark
&
\cmark
&
history representation
&
\xmark
&
\cmark
&
\cmark
\\

\addlinespace[3pt]

\rowcolor{gray!15}
\textbf{This work}
&
\cmark
&
\cmark
&
\cmark
&
lifted history
&
sufficiently small step size (Thm.~\ref{thm:local-preservation})
&
\cmark\ on $\tilde{\mathcal K}_p$
&
\cmark\ under Assms.~\ref{ass:history} and~\ref{ass:reachability}
\\

\bottomrule

\end{tabularx}
\end{table*}

Table~\ref{tab:landscape} makes explicit that our guarantees require the stated lifted policy class and closed-loop reachability. Per-iteration stability and reachability are preserved by the sufficiently small step sizes characterized in Theorem~\ref{thm:local-preservation}.
\vspace{0.2cm}

\noindent \textbf{Notation.} For a vector $x$, $\|x\|_2$ denotes its Euclidean norm. For a random vector $\xi$, we define its root-mean-square norm as $\|\xi\|_{L_2}:=(\mathbb E[\|\xi\|_2^2])^{1/2}$. For a matrix $M$, $\|M\|$ and $\|M\|_F$ denote its spectral and Frobenius norms, respectively; $\rho(M)$ denotes its spectral radius, $\underline{\sigma}(M)$ its smallest singular value, and $\operatorname{Tr}(M)$ its trace.

\section{Problem Formulation}
\label{sec:LQG}
Consider the discrete, linear time-invariant  system  
\begin{equation}
  \label{LQG}
  \begin{aligned}
    x_{t+1} &= A x_t + B u_t + w_t, \\
    y_t &= C x_t + v_t, \text{ for } t = 0,1,2,\ldots,
  \end{aligned}
\end{equation}
where $x_t \in \mathbb{R}^{n_x}$ denotes the state, $u_t \in \mathbb{R}^{n_u}$ the control input, and $y_t \in \mathbb{R}^{n_y}$ the measured output of the system at time $t$. Let $w_t$ and $v_t$ denote the process and measurement noise at time $t$, respectively. The process and measurement noise are assumed to be i.i.d. with $w_t \overset{\text{i.i.d.}}{\sim} \mathcal{N}(0, W)$ and $v_t \overset{\text{i.i.d.}}{\sim} \mathcal{N}(0, V)$, where $W \succeq 0$ and $V \succ 0$. Moreover, the noise sequences $\{w_t\}_{t \geq 0}$ and $\{v_t\}_{t \geq 0}$ are assumed to be mutually independent.
\begin{assumption} \label{assumption:controllable_observable}
The pair $(A,B)$ is assumed to be controllable, $(A,C)$ observable, and $(A,W^{1/2})$ stabilizable, ensuring that the stabilizing steady-state Kalman filter exists.
\end{assumption}
The LQG objective is to design a control sequence $\{u_t\}_{t \geq 0}$ that minimizes the  ergodic quadratic cost:
\begin{equation}
\label{eq:cost}
J :=  \limsup_{T \to \infty}\frac{1}{T}\mathbb{E} \left[ \sum_{t=0}^{T-1}  y_t^\top Q y_t + u_t^\top R u_t  \right], \text{ s.t. \eqref{LQG}},
\end{equation}
where $Q \succeq 0$ and $R \succ 0$ are the output and control input weight matrices, respectively. The expectation is taken with respect to the control and noise sequences. It is well-known \cite{anderson2007optimal} that given the system matrices $A$, $B$ and $C$, the control sequence that minimizes \eqref{eq:cost} admits the form 

\begin{equation}
\label{eq:control-law}
u_t = K^\star\hat{x}_t, \text{ for } t = 0,1,2,\ldots,
\end{equation}

where $\hat{x}_t$ denotes the state estimate described by
\begin{align}
\label{eq:sys}
\hat{x}_{t+1} &= \tilde{A}\hat{x}_t + \tilde{B} u_t + L y_{t+1}, 
\end{align}
with $\tilde{A} = \left(I - LC\right) A$ and $
\tilde{B} = \left(I - LC\right) B$, where $L$ is the steady-state Kalman filter gain, and $K^\star$ is the optimal LQG controller. In particular, $K^\star$ and $L$ are given by 
\begin{align*}
K^\star = -(R+B^\top P B)^{-1}B^\top PA, \; L = \tilde{\Sigma}C^\top (C\tilde{\Sigma}C^\top + V)^{-1},
\end{align*}
where $P = \texttt{dare}(A,B,\tilde{Q},R)$ with cost matrix $\tilde{Q} = C^\top QC$, and covariance $\tilde{\Sigma} := \texttt{dare}(A^\top, C^\top, W,V)$.

An alternative approach for solving \eqref{eq:cost} \emph{without} resorting to system identification is to estimate the policy gradients with respect to $K$ and perform PG updates in ``controller space''. 

However, the convergence of PG when minimizing \eqref{eq:cost} over the set of stabilizing controllers subject to \eqref{LQG} is hindered by two main challenges: 1) the cost \eqref{eq:cost} does not satisfy gradient dominance under the controller parameterization given by \eqref{eq:control-law} and \eqref{eq:sys} \cite{mohammadi_lack_2021}, i.e., under such a controller parameterization there is no guarantee that a stationary solution is globally optimal and that a small gradient norm would imply closeness to the optimal controller. 2) the set of stabilizing controllers itself is non‐convex with at most two disjoint (though nearly identical) connected components~\cite{tang_analysis_2023,zheng_escaping_2022}. This ``non-benign'' landscape is composed of multiple stationary solutions.

Motivated by these challenges and inspired by the representation proposed in \cite{skelton_data-based_1994}, we adopt this alternative parameterization of the set of stabilizing controllers. For the lifted policy class and reachability condition stated below, the infinite-history cost \emph{is} gradient dominated, yielding conditional convergence guarantees. We emphasize that the finite-history implementations incur a geometrically decaying approximation error.

\subsection{History Representation}
Let $p \in \mathbb{N}$ denote the history length, and $u_{t,p} \in \mathbb{R}^{n_up}$, $y_{t,p} \in \mathbb{R}^{n_yp}$ be the input and output history at time $t$, i.e.,
\begin{equation*}
\label{def:z_t}
u_{t,p} = \begin{bmatrix} u_{t-1}\\ \vdots \\ u_{t-p} \end{bmatrix} \text{ and } y_{t,p} = \begin{bmatrix} y_{t} \\ \vdots \\ y_{t-p+1} \end{bmatrix}, 
\end{equation*}
where we define $z_{t,p} = [u^\top_{t,p} \;\ y^\top_{t,p}]^\top \in \mathbb{R}^{p(n_u+n_y)}$ as the input-output history data. {\color{black}
Let $\nu_r$ denote the reachability index of $(\tilde A,L)$, i.e., the smallest integer for which $[L\;\tilde A L\;\cdots\;\tilde A^{\nu_r-1}L]$ has full row rank. For the remainder of the paper, we select $p\geq\nu_r$. The effect of the choice of $p$ is analyzed in the sequel.

By unrolling the steady-state Kalman filter over $p$ steps, we obtain
\begin{equation}
\label{eq:truncated-history}
\hat x_t=S_pz_{t,p}+e_{t,p},
\end{equation}
with $S_p:=\begin{bmatrix}\mathcal F_{u,p}&\mathcal F_{y,p}\end{bmatrix},\; e_{t,p}:=\tilde A^p\hat x_{t-p},$ and 
\begin{align*}
\mathcal F_{u,p}&=[\tilde B\;\tilde A\tilde B\;\cdots\;\tilde A^{p-1}\tilde B]\\
\mathcal F_{y,p}&=[L\;\tilde A L\;\cdots\;\tilde A^{p-1}L]
\end{align*}
We  make two observations: i)  $S_p$ does not depend on a control policy, and ii)  $e_{t,p}$ denotes the truncation error in the state estimate.
\begin{assumption}
\label{ass:history}
The history length satisfies $p\geq\nu_r$, $S_p$ has full row rank, and lifted policies are restricted to
$$
\tilde{\mathcal K}_p:=\{\tilde K=KS_p:\ \rho(A+BK)<1\}.
$$
Equivalently, $\tilde K=\tilde KS_p^\dagger S_p$ and the underlying state-feedback gain is $K=\tilde KS_p^\dagger$. {\color{black}We take $p$ sufficiently large that the delayed finite-history closed loop satisfies the small-gain condition in Lemma~\ref{lem:finite-history-stability}. The resulting cost and gradient errors are quantified in Lemma~\ref{lem:cost-gradient-truncation}.}
\end{assumption}

We note that the above assumption isolates the policy class for which the lifted LQR argument is valid. It excludes arbitrary null-space components, which can change the actual input $\tilde Kz_{t,p}$ without changing $\tilde KS_p^\dagger$. 
\vspace{0.2cm}

\noindent \textbf{Initialization of the history.} Each rollout starts from a random $x_0$ with finite second moment, independent of the process and measurement noises. A fixed bounded sequence $u_0,\ldots,u_{p-1}$ is applied and the outputs $y_0,\ldots,y_p$ are stored. The first policy action is taken at time $p$, when $z_{p,p}=[u_{p-1}^\top\ \cdots\ u_0^\top\ y_p^\top\ \cdots\ y_1^\top]^\top$ is available. All finite-horizon rollouts use the same initial-state distribution. For a stabilizing policy, this initialization does not affect the ergodic cost as its contribution vanishes as the horizon grows.

{\color{black}
\begin{lemma}
\label{lem:truncation}
There exist constants $C_1>0$ and $\rho_1\in(0,1)$ such that
\begin{align*}
\|\tilde A^p\|\leq C_1\rho_1^p
\end{align*}
and
\begin{align*}
\|e_{t,p}\|_{L_2}\leq C_1\rho_1^p\|\hat x_{t-p}\|_{L_2}.
\end{align*}
Then, for $\tilde K=KS_p\in\tilde{\mathcal K}_p$, it holds that
\begin{align*}
\|\tilde Kz_{t,p}-K\hat x_t\|_{L_2}
\leq C_1\|K\|\rho_1^p\|\hat x_{t-p}\|_{L_2}.
\end{align*}
\end{lemma}

\begin{proof}
The steady-state Kalman predictor matrix $\tilde A$ is Schur under Assumption~\ref{assumption:controllable_observable}. Hence $\|\tilde A^p\|\leq C_1\rho_1^p$. The remaining claims follow from~\eqref{eq:truncated-history} and $\tilde K=KS_p$.
\end{proof}
}
\textcolor{black}{Figure~\ref{fig:truncation-error} provides a numerical
illustration of Lemma~1. We note that the empirical mean-square truncation
error decays geometrically as the history length $p$ increases.}
\begin{figure}[t]
    \centering
\includegraphics[width=1\columnwidth]
    {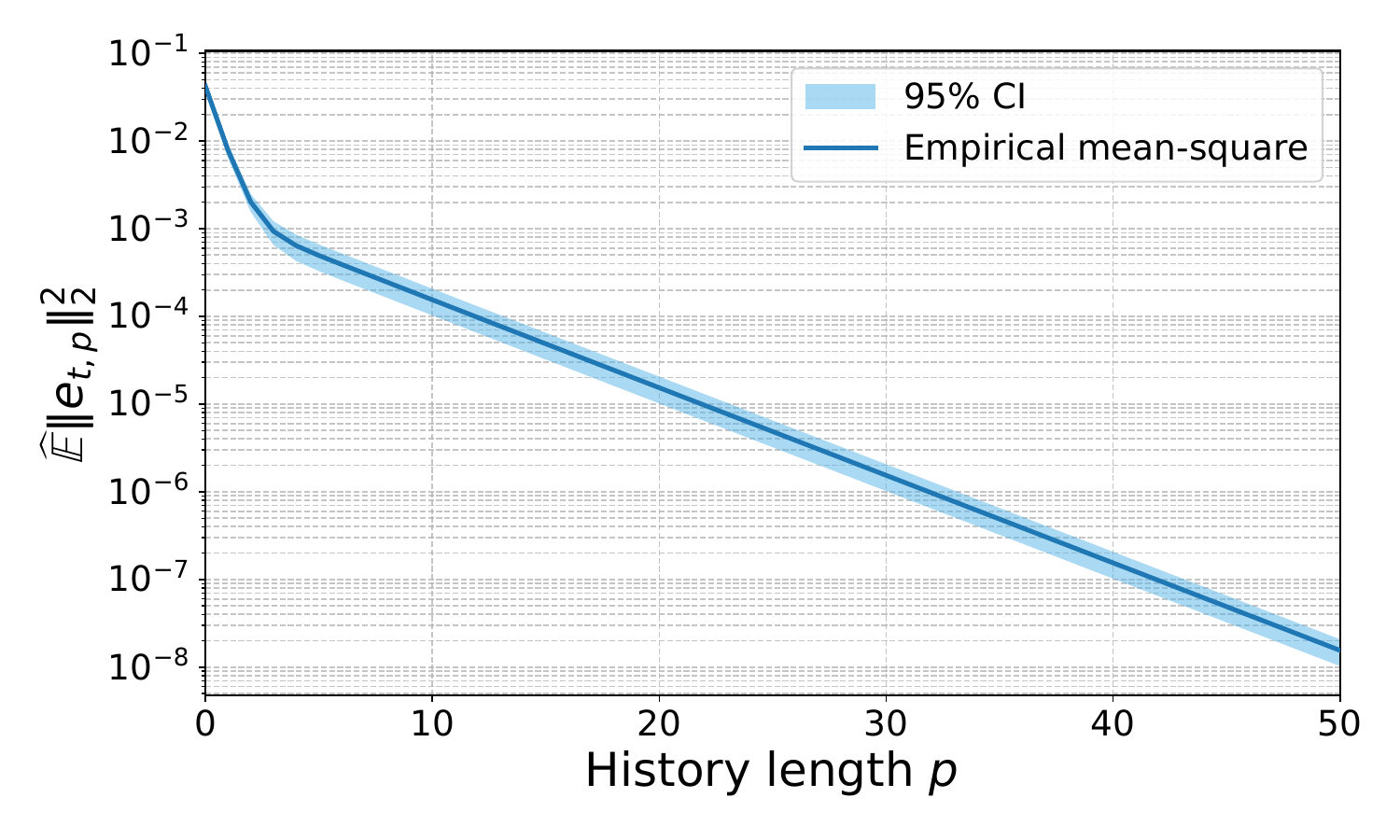}
\caption{
Mean-square history-truncation error as a function of the
history length $p$. The empirical error
$\mathbb{E}\|e_{t,p}\|_2^2$ decays geometrically with $p$,
consistent with the bound in Lemma~1. The shaded region
is the $95\%$ confidence interval.
}
\label{fig:truncation-error}
\end{figure}

{\color{black}
For a fixed controller $K$, let $(x_t^p,\hat x_t^p,u_t^p)$ denote the state, Kalman estimate, and input of the finite-history closed loop
\begin{align*}
u_t^p=KS_pz_{t,p}^p,
\end{align*}
where $z_{t,p}^p$ is the input-output history generated by this closed loop. For this trajectory, $e_{t,p}:=\tilde A^p\hat x_{t-p}^p$. We define its augmented state as follows: $\xi_t^p:=[(x_t^p)^\top\;(\hat x_t^p)^\top]^\top$. In addition, let
\begin{align*}
H&:=\begin{bmatrix}0&I\end{bmatrix},\;
D:=\begin{bmatrix}B\\B\end{bmatrix},\;
D_K:=DK,\\
F_K&:=\begin{bmatrix}A&BK\\LCA&\tilde A+BK\end{bmatrix},\;
G:=\begin{bmatrix}I&0\\LC&L\end{bmatrix}.
\end{align*}

\begin{lemma}[Finite-history stability]
\label{lem:finite-history-stability}
Fix a stabilizing controller $K$ and a sufficiently small closed neighborhood of it. There exist constants $C_2,C_3>0$ and $\rho_2\in(0,1)$, independent of $p$, such that
\begin{align*}
\|F_K^s\|\leq C_2\rho_2^s,\; \|D_K\|\leq C_3
\end{align*}
uniformly on this neighborhood. If
\begin{align}
\label{eq:history-small-gain}
\frac{C_1C_2C_3}{1-\rho_2}\rho_1^p<1,
\end{align}
then the finite-history closed loop is Schur. Moreover, if the left-hand side of~\eqref{eq:history-small-gain} is bounded above by a fixed number smaller than one, then there exists $C_4>0$, independent of $p$, such that $\|\xi_t^p\|_{L_2}\leq C_4$ in steady state.
\end{lemma}

\begin{proof}
The separation principle provides a similarity transformation under which the eigenvalues of $F_K$ are those of $A+BK$ and $\tilde A$. Thus, $F_K$ is Schur, and the stated bounds hold uniformly after choosing the neighborhood sufficiently small.
Let $\zeta_{t+1}:=[w_t^\top\;v_{t+1}^\top]^\top$. As $e_{t,p}=\tilde A^pH\xi_{t-p}^p$, the finite-history dynamics satisfy
\begin{align}
\label{eq:finite-history-delayed-system}
\xi_{t+1}^p
= F_K\xi_t^p-D_K\tilde A^pH\xi_{t-p}^p+G\zeta_{t+1}.
\end{align}
The corresponding homogeneous delayed system is then obtained by setting $\zeta_{t+1}=0$. To express it as a standard linear system, we define the state
\begin{align*}
\chi_t^p:=
\begin{bmatrix}
\xi_t^p\\
\xi_{t-1}^p\\
\vdots\\
\xi_{t-p}^p
\end{bmatrix}.
\end{align*}
Then, we can write
\begin{align*}
\chi_{t+1}^p=\mathcal F_{K,p}\chi_t^p,
\end{align*}
where the delayed companion matrix is given by
\begin{align*}
\mathcal F_{K,p}
:=
\begin{bmatrix}
F_K & 0 & \cdots & 0 & -D_K\tilde A^pH\\
I   & 0 & \cdots & 0 & 0\\
0   & I & \cdots & 0 & 0\\
\vdots & & \ddots & & \vdots\\
0 & 0 & \cdots & I & 0
\end{bmatrix}.
\end{align*}
Let us then suppose, toward a contradiction, that $\mathcal F_{K,p}$ has an eigenvalue $\lambda$ with $|\lambda|\geq1$. Its corresponding
eigenvector has the form
\begin{align*}
\begin{bmatrix}
v^\top &
\lambda^{-1}v^\top &
\cdots &
\lambda^{-p}v^\top
\end{bmatrix}^\top
\end{align*}
for some nonzero leading block $v$. The first block row of the eigenvalue equation therefore yields
\begin{align*}
(\lambda I-F_K)v=-\lambda^{-p}D_K\tilde A^pHv.
\end{align*}

Then, we obtain
\begin{align*}
\|(\lambda I-F_K)^{-1}\|
&\leq\sum_{s=0}^{\infty}|\lambda|^{-s-1}\|F_K^s\|\\
&\leq\frac{C_2}{1-\rho_2}.
\end{align*}
By leveraging Lemma~\ref{lem:truncation}, we then have 
\begin{align*}
\|v\|\leq\frac{C_1C_2C_3}{1-\rho_2}\rho_1^p\|v\|,
\end{align*}
which contradicts~\eqref{eq:history-small-gain}. Thus, the delayed companion matrix is Schur. 

Moreover, from~\eqref{eq:finite-history-delayed-system}, we can write
\begin{align*}
\sup_t\|\xi_t^p\|_{L_2}
&\leq C_2\|\xi_0^p\|_{L_2}+\frac{C_2\|G\|}{1-\rho_2}\sup_t\|\zeta_t\|_{L_2}\\
&+\frac{C_1C_2C_3}{1-\rho_2}\rho_1^p\sup_t\|\xi_t^p\|_{L_2}.
\end{align*}
The last term is absorbed into the left-hand side under the fixed strict margin in~\eqref{eq:history-small-gain}, proving the uniform second-moment bound.
\end{proof}
}

{\color{black}Let $(x_t^\infty,\hat x_t^\infty,u_t^\infty)$ denote the infinite-history closed-loop trajectory generated by $u_t^\infty=K\hat x_t^\infty$. For the trajectory comparison below, the finite- and infinite-history loops use the same initial condition and noise realizations. Let $J_\infty(K)$ denote the corresponding infinite-history cost, and let $J_p(K)$ denote the cost generated by $u_t^p=KS_pz_{t,p}^p$.}

{\color{black}
\begin{lemma}
\label{lem:cost-gradient-truncation}
Fix a stabilizing controller $K$ and a sufficiently small closed neighborhood around it. If $p$ satisfies~\eqref{eq:history-small-gain} with a fixed strict margin, then there exist finite constants $C_5,C_6>0$, independent of $p$, such that
\begin{align*}
|J_p(K)-J_\infty(K)|&\leq C_5\rho_1^p,\\
\|\nabla J_p(K)-\nabla J_\infty(K)\|_F&\leq C_6\rho_1^p.
\end{align*}
where $\rho_1$ is defined in Lemma~\ref{lem:truncation}.
\vspace{0.2cm}
\end{lemma}

\begin{proof}
Consider the finite- and infinite-history closed loops under the same initial condition and noise realizations. For $j\in\{p,\infty\}$, we recall
\begin{align*}
\xi_t^j:=
\begin{bmatrix}x_t^j\\ \hat x_t^j\end{bmatrix},\;
d_t:=\xi_t^p-\xi_t^\infty
=\begin{bmatrix}x_t^p-x_t^\infty\\ \hat x_t^p-\hat x_t^\infty\end{bmatrix}.
\end{align*}
Then, under $u_t^\infty=K\hat x_t^\infty$, the augmented state $\xi_t^{\infty}$ satisfies
\begin{align*}
\xi_{t+1}^\infty&=F_K\xi_t^\infty+G\zeta_{t+1},
\end{align*}
where we used the fact that $\tilde B+LCB=B$. Moreover, $S_pz_{t,p}^p=\hat x_t^p-e_{t,p}$ yields
\begin{align*}
u_t^p-u_t^\infty
&=K(\hat x_t^p-\hat x_t^\infty)-Ke_{t,p}.
\end{align*}
By subtracting the two recursions, we obtain
\begin{align*}
d_{t+1}=F_Kd_t-D_Ke_{t,p},\;d_0=0.
\end{align*}
Lemma~\ref{lem:finite-history-stability} first establishes $\|\xi_t^p\|_{L_2}\leq C_4$. Thus, Lemma~\ref{lem:truncation} and $\hat x_t^p=H\xi_t^p$ give $\|e_{t,p}\|_{L_2}\leq C_1C_4\rho_1^p$. By using $\|F_K^s\|\leq C_2\rho_2^s$ in the preceding recursion, we have 
\begin{align*}
\|d_t\|_{L_2}
&\leq C_3\sum_{s=0}^{t-1}\|F_K^s\|\|e_{t-1-s,p}\|_{L_2}\\
&\leq\frac{C_1C_2C_3C_4}{1-\rho_2}\rho_1^p.
\end{align*}
Hence, the augmented-state and input differences are of order $\rho_1^p$. Let us define the stage cost as follows:
\begin{align*}
\ell_t^j := (y_t^j)^\top Qy_t^j+(u_t^j)^\top Ru_t^j, \qquad j\in\{p,\infty\}.
\end{align*}
As the two closed loops use the same measurement-noise realization, we have
\begin{align*}
y_t^p-y_t^\infty = C(x_t^p-x_t^\infty).
\end{align*}
In addition, we can write
\begin{align*}
u_t^p-u_t^\infty = K(Hd_t-e_{t,p}).
\end{align*}
Therefore, the preceding bounds on $d_t$ and $e_{t,p}$ imply
\begin{align*}
\|y_t^p-y_t^\infty\|_{L_2}
+
\|u_t^p-u_t^\infty\|_{L_2}
\leq C_5\rho_1^p,
\end{align*}

Lemma~\ref{lem:finite-history-stability}
and the stability of the infinite-history closed loop also give uniform bounds on the second moments of $y_t^j$ and $u_t^j$. By using the identity
\begin{align*}
a^\top Ma-b^\top Mb=(a-b)^\top M(a+b)
\end{align*}
and applying Cauchy-Schwarz, we obtain
\begin{align*}
\mathbb E\left[|\ell_t^p-\ell_t^\infty|\right]
&\leq
\|Q\|
\|y_t^p-y_t^\infty\|_{L_2}
\|y_t^p+y_t^\infty\|_{L_2}\\
&+\|R\|\|u_t^p-u_t^\infty\|_{L_2}
\|u_t^p+u_t^\infty\|_{L_2}\\
&\leq C_5\rho_1^p.
\end{align*}
Averaging over time and taking the limit superior yields
\begin{align*}
|J_p(K)-J_\infty(K)|
\leq C_5\rho_1^p.
\end{align*}

For the gradient bound, fix $\Delta K$ with $\|\Delta K\|_F=1$ and let $q_t^p$ and $q_t^\infty$ denote the directional derivatives of $\xi_t^p$ and $\xi_t^\infty$ with respect to $K$. By differentiating the two recursions provides
\begin{align*}
q_{t+1}^p&=F_Kq_t^p+D\Delta K(H\xi_t^p-e_{t,p})-D_K\tilde A^pHq_{t-p}^p,\\
q_{t+1}^\infty&=F_Kq_t^\infty+D\Delta KH\xi_t^\infty.
\end{align*}
The same small-gain condition from Lemma \ref{lem:finite-history-stability}  bounds $q_t^p$ uniformly. Then, by subtracting these recursions and using the bounds on $d_t$, $e_{t,p}$, and $\tilde A^p$, we have
\begin{align*}
\sup_t\|q_t^p-q_t^\infty\|_{L_2}\leq C_6\rho_1^p.
\end{align*}
By differentiating the quadratic stage cost, applying Cauchy-Schwarz, and taking the supremum over $\|\Delta K\|_F=1$ therefore yields
\begin{align*}
\|\nabla J_p(K)-\nabla J_\infty(K)\|_F\leq C_6\rho_1^p.
\end{align*}
For the lifted parameterization, we define $J_j(\tilde K):=J_j(\tilde KS_p^\dagger)$ for $j\in\{p,\infty\}$. The chain rule gives $\nabla_{\tilde K}J_j=\nabla_KJ_jS_p^{\dagger\top}$. Since $S_pS_p^\top$ is nondecreasing in $p$, $\|S_p^\dagger\|$ is uniformly bounded for $p\geq\nu_r$, so the same result also holds for $\nabla_{\tilde K}J_j$.
\end{proof}
}

\begin{figure}[t]
\centering
\includegraphics[width=0.97\columnwidth]
    {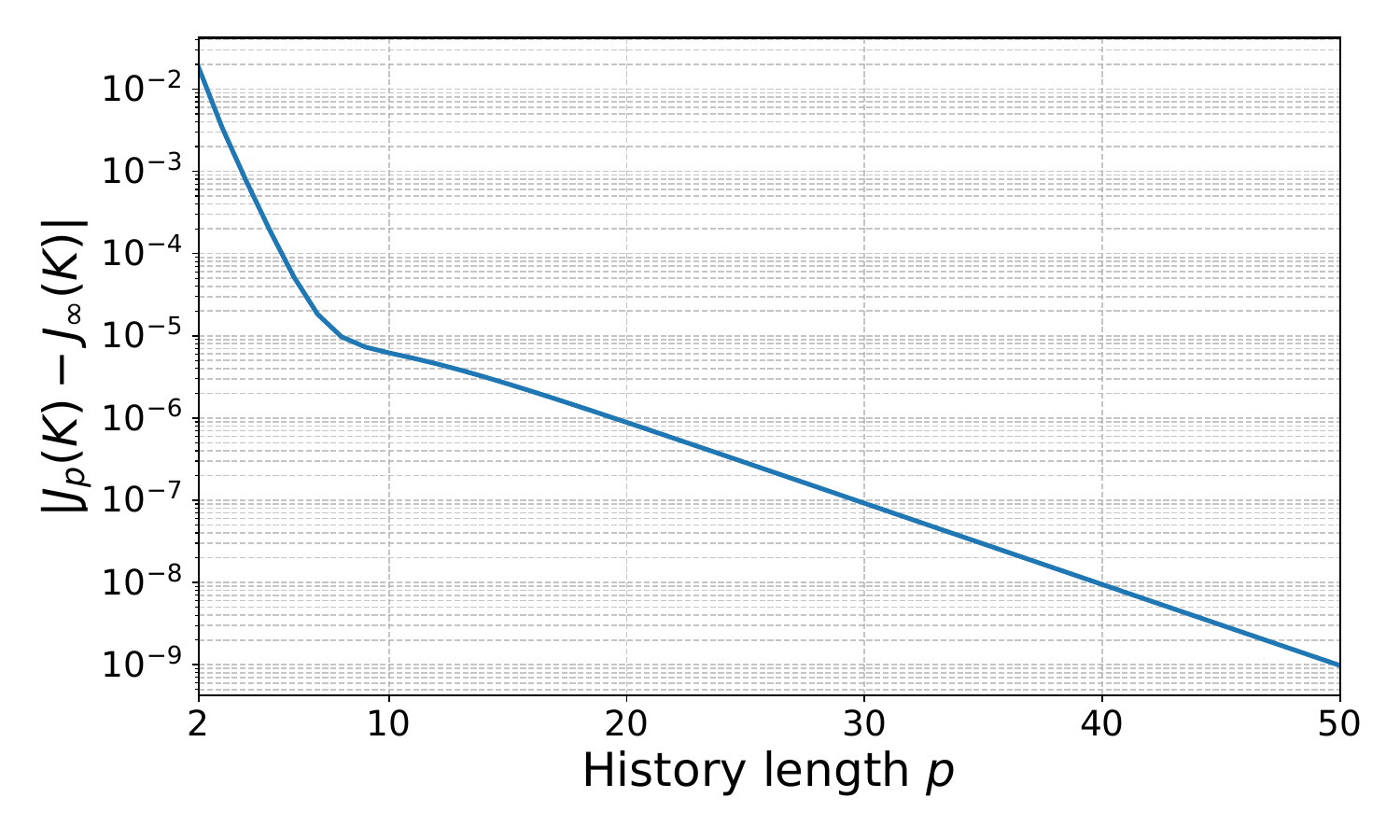}
\caption{
Approximation error between the finite- and infinite-history
costs as a function of the history length $p$. The quantity
$|J_p(K)-J_\infty(K)|$ is evaluated at the fixed stabilizing
controller $K=0.75K^\star$ and exhibits the geometric decay
predicted by Lemma~\ref{lem:cost-gradient-truncation}.
}
\label{fig:cost-approximation-error}
\end{figure}
\begin{figure}[t]
\centering
\includegraphics[width=0.97\columnwidth]
    {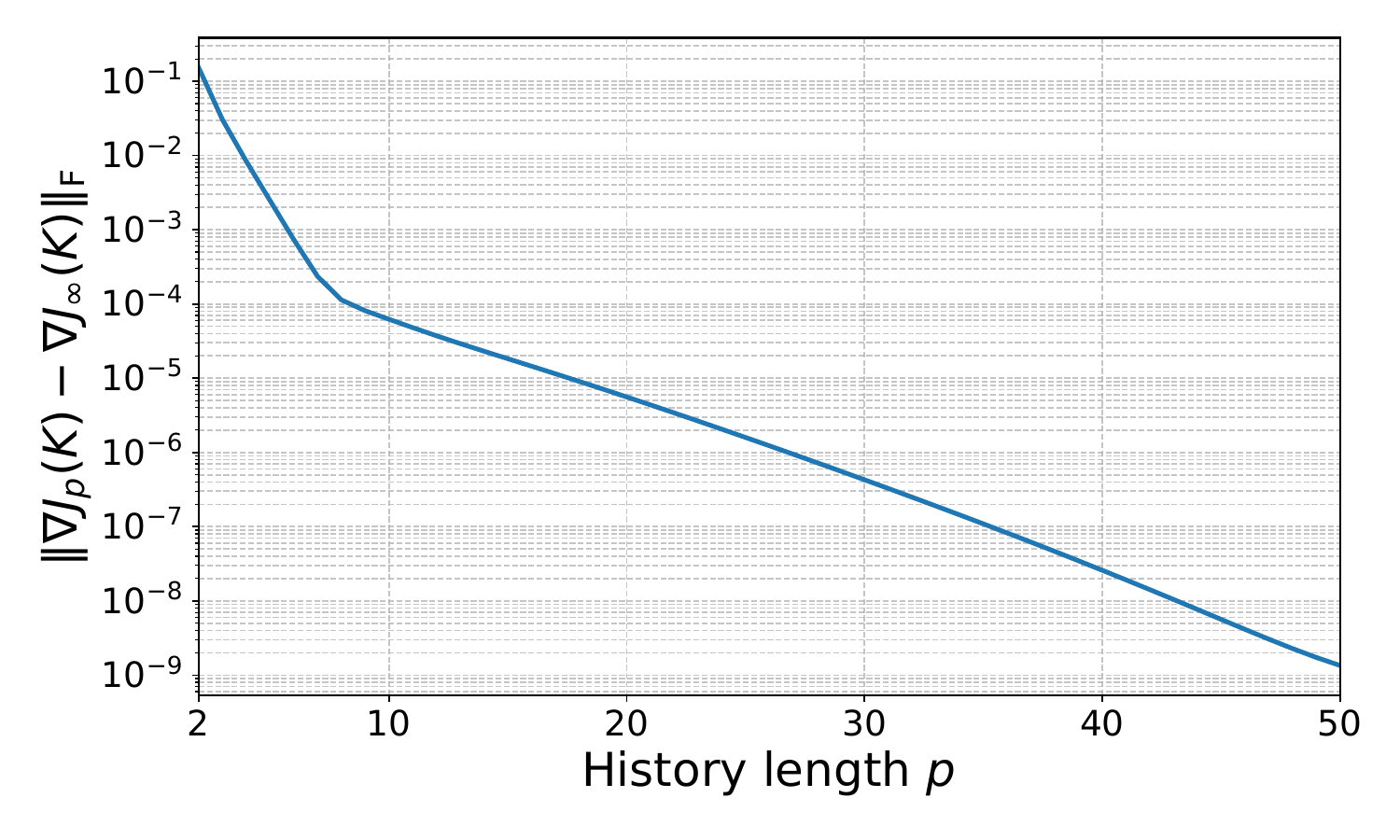}
\caption{
Approximation error between the finite- and infinite-history
policy gradients as a function of the history length $p$.
The error
$\|\nabla J_p(K)-\nabla J_\infty(K)\|_F$, evaluated at
$K=0.75K^\star$, decreases geometrically with $p$, in
agreement with Lemma~\ref{lem:cost-gradient-truncation}.
}
\label{fig:gradient-approximation-error}
\end{figure}

\textcolor{black}{Figures~\ref{fig:cost-approximation-error}
and~\ref{fig:gradient-approximation-error} numerically illustrate
the bounds in Lemma~\ref{lem:cost-gradient-truncation}, showing the decay of the cost and gradient approximation errors as the history length $p$ increases.}

The optimal lifted history gain is $\tilde K_p^\star=K^\star S_p$. In the sequel, we retain the explicit notation $S_p$, $\tilde{\mathcal K}_p$, $J_\infty$, and $J_p$ to distinguish the infinite-history and finite-history problems. Lemma~\ref{lem:cost-gradient-truncation} relates their costs and gradients up to a truncation error in the order of $\mathcal O(\rho_1^p)$.}

\subsection{Policy Gradient for the LQG problem (PG-LQG)} \label{subsec:PG-LQG model-base}

Let us consider the \emph{model-based} PG setting, 
i.e., we have access to $(A,B,C,Q,R,W,V)$. We define the LQG cost and gradient with respect to the lifted controller $\tilde{K} \in \tilde{\mathcal K}_p$. Next we show that, under Assumptions~\ref{ass:history} and~\ref{ass:reachability}, the infinite-history lifted cost satisfies gradient dominance and smoothness. On the other hand the finite-history cost is controlled by Lemma~\ref{lem:cost-gradient-truncation}.

For any stabilizing controller $\tilde{K} \in \tilde{\mathcal K}_p$, the infinite-history cost $J_\infty(\tilde{K})$ can be expressed as follows:
\begin{equation}
\label{eq:lqg_cost_tildeK}
J_\infty(\tilde{K}) = \operatorname{Tr}\left(P_{\tilde{K}}\,\Sigma_{\nu}\right) + \operatorname{Tr}\left(\tilde{Q}\,(I - L\,C) \tilde{\Sigma}\right) + \operatorname{Tr}\left(Q\,V\right).
\end{equation}
 Here, $P_{\tilde{K}} \succeq 0$ denotes the solution of the discrete-time closed-loop Lyapunov equation  
\begin{align*}\hspace{-0.25cm}P_{\tilde{K}} \hspace{-0.05cm}=  \hspace{-0.05cm} \tilde{Q} + S_p^{\dagger\top}\tilde{K}^\top R\tilde{K}S_p^{\dagger} + (A + B\,\tilde{K}S_p^{\dagger})^\top P_{\tilde{K}}(A + B\tilde{K}S_p^{\dagger})\end{align*}
and $\Sigma_{\nu} := L\,\Big(C\tilde{\Sigma}C^\top + V\Big)\,L^\top$. The cost in \eqref{eq:lqg_cost_tildeK} is similar to the one presented in \cite{ziemann_how_2022}, where the difference is the rescaled cost matrix $\tilde{Q} = C^\top Q C$ and the third term that arises by defining the ergodic cost \eqref{eq:cost} w.r.t. $y_t$ instead of $x_t$.

The second and third terms in \eqref{eq:lqg_cost_tildeK} do not depend on the controller $\tilde{K}$ and thus have no impact on $\nabla J_\infty(\tilde{K})$.

\begin{lemma}\label{lemma:2}
Given $\tilde{K} \in \tilde{\mathcal K}_p$, the gradient $\nabla J_\infty(\tilde{K})$ is
\begin{align*}
\nabla J_\infty(\tilde{K}) := 2\, E_{\tilde{K}}\, \Sigma_{\tilde{K}}\, S_p^{\dagger \top},
\end{align*}
where $E_{\tilde{K}} := (R + B^{\top} P_{\tilde K} B)\, \tilde{K}\, S_p^{\dagger} + B^{\top} P_{\tilde{K}} A,$ and the covariance $\Sigma_{\tilde{K}} = \texttt{dlyap}(A+B\tilde{K}S_p^{\dagger}, \Sigma_{\nu})$.

\begin{proof}
The proof follows directly from \cite[Lemma~6]{ziemann_how_2022} with $K = \tilde{K} S_p^{\dagger}$ and $\nabla_K J_\infty := 2\, E_{\tilde{K}}\, \Sigma_{\tilde{K}}$. Applying the chain rule through $K$ gives $\nabla J_\infty(\tilde{K}) = \nabla_K J_\infty \, (S_p^{\dagger})^{\top} = 2\, E_{\tilde{K}}\, \Sigma_{\tilde{K}}\, S_p^{\dagger \top}\!.$
\end{proof}
\end{lemma}

We note that the expression of the gradient in Lemma~\ref{lemma:2} is analogous to the gradient of the LQR cost in \cite{fazel_global_2019}, except by the presence of $S_p^{\dagger \top}$ right multiplying the state covariance. Moreover, the cost function $J_\infty(\tilde{K})$ in \eqref{eq:lqg_cost_tildeK} consists of three terms, of which only the first depends on $\tilde{K}$ (as in the LQR problem). Therefore, it is natural to conclude that, under the history representation, the LQG cost with respect to $\tilde{K}$ is gradient dominated. We formalize this claim in Lemma \ref{lemma:gradient dominance}. First, we define the sublevel set of stabilizing controllers.

\begin{definition} Let $\alpha > 0$ be a positive scalar and fix a stabilizing controller $\tilde{K}^\prime \in \tilde{\mathcal K}_p$. The sublevel set of stabilizing controllers $\tilde{\mathcal K}_{p,\tilde K^\prime,\alpha} \subseteq \tilde{\mathcal K}_p$ is denoted by
\begin{equation*}
\tilde{\mathcal K}_{p,\tilde K^\prime,\alpha}\hspace{-0.1cm}:=\hspace{-0.05cm}\left\{\hspace{-0.05cm} \tilde{K} \middle| J_\infty(\tilde{K}) \hspace{-0.05cm}-\hspace{-0.05cm} J_\infty(\tilde{K}_p^\star) \hspace{-0.05cm}\leq\hspace{-0.05cm} \alpha \left( J_\infty(\tilde{K}^\prime)\hspace{-0.05cm} -\hspace{-0.05cm} J_\infty(\tilde{K}_p^\star) \right) \hspace{-0.05cm}\right\}.
\end{equation*} 
\end{definition}

Later, with a slight abuse of notation we denote $\tilde{\mathcal K}_{p,\alpha}$ as the sublevel set of stabilizing controllers with respect to some initial stabilizing controller $\tilde{K}^\prime = \tilde{K}_0 \in \tilde{\mathcal K}_p$.

We emphasize that although $\widetilde K_p^\star:=K^\star S_p$ depends on $p$ through its lifted representation, it is the minimizer of the infinite-history
cost over $\widetilde{\mathcal K}_p$. Indeed, $\widetilde K_p^\star S_p^\dagger=K^\star$. It need not minimize the finite-history cost $J_p$.

\begin{lemma}[Gradient Dominance]
\label{lemma:gradient dominance}
Given a stabilizing controller $\tilde{K} \in \tilde{\mathcal K}_{p,\alpha}$ and $(A + B \tilde{K} S_p^{\dagger},\Sigma_\nu^{1/2}) $ is controllable, it holds that
\begin{align*}
\hspace{-0.1cm}J_\infty(\tilde{K}) \hspace{-0.02cm}-\hspace{-0.02cm} J_\infty(\tilde{K}_p^\star) \hspace{-0.05cm}&\le\hspace{-0.05cm} \frac{\|\Sigma_{\tilde K_p^\star}\|\hspace{-0.02cm}\|S_p\|^2}{4\,\underline{\sigma}(R)\,\underline{\sigma}^2(\Sigma_{\tilde{K}})}  \|\nabla J_\infty(\tilde{K})\|_F^2\hspace{-0.05cm}\\
&=:\hspace{-0.05cm} \mu_{\texttt{PL}}\|\nabla J_\infty(\tilde{K})\|_F^2,
\end{align*}
where $\Sigma_{\tilde K_p^\star}$ is the state covariance matrix under the optimal controller $\tilde{K}_p^\star$.
\end{lemma}

\begin{proof} The proof follows  from \cite[Lemma 11]{fazel_global_2019} by using our definition for LQG controller $K = \tilde{K}S_p^{\dagger}$. For the completeness we begin the proof by writing the cost different between $\tilde{K}$ and $\tilde K^\star_p$ as follows:  
\begin{align*}
J_\infty(\tilde K)-J_\infty(\tilde K_p^\star)
=-\mathbb E_{\tilde K_p^\star}
\left[A_{\tilde K}
(\hat x_t^\star,u_t^\star)\right].
\end{align*}
where $A_{\tilde{K}}(\hat x_t^\star,u_t^\star)$ denotes the advantage function\footnote{$A_{\tilde K}\bigl(\hat x_t^\star,u_t^\star\bigr)$ is defined as in \cite[Section C.1]{fazel_global_2019}.} (i.e., the difference between the state-action value function and the value function) of playing the controller $\tilde{K}$ given the sequence of state and inputs $\{\hat x^\star_t, u^\star_t\}_{t\geq 0}$ generated by the optimal lifted controller $\tilde{K}_p^\star$. By the quadratic form of the advantage function, it is routine to obtain the following lower bound 
\begin{align*}
A_{\widetilde K}(\hat x_t^\star,u_t^\star)
\geq -\operatorname{Tr}\!\left(
\hat x_t^\star\hat x_t^{\star\top}
E_{\widetilde K}^\top
(R+B^\top P_{\widetilde K}B)^{-1}
E_{\widetilde K}
\right).
\end{align*}
Taking the expectation and summing over time, we have
\begin{align*}
J_\infty(\tilde K)-J_\infty(\tilde K_p^\star)
\le\Tr\Bigl(\Sigma_{\tilde K_p^\star}\,
E_{\tilde K}^\top\,(R + B^\top P_{\tilde K}B)^{-1}\,E_{\tilde K}\Bigr).
\end{align*}
We use the inequality $\Tr(M_1 M_2)\le\|M_1\|\Tr(M_2)$ for any matrices $M_1$ and $M_2$ of compatible dimensions, with $M_2$ being positive semidefinite, and with $\|(R+B^\top P_{\tilde K}B)^{-1}\|\le1/\underline\sigma(R)$, to obtain
\begin{align} \label{eq:cost different - PL}
 J_\infty(\tilde K)-J_\infty(\tilde K_p^\star)
\leq \frac{\|\Sigma_{\tilde K_p^\star}\|}{\underline\sigma(R)}
\Tr \bigl(E_{\tilde K}^\top E_{\tilde K}\bigr). 
\end{align}
Finally, by Lemma \ref{lemma:2} we have  
$\nabla J_\infty(\tilde K)=2\,E_{\tilde K}\Sigma_{\tilde K}\,S_p^{\dagger \top}$, which implies
\begin{align*}
\Tr\bigl(E_{\tilde K}^\top E_{\tilde K}\bigr)
=\frac{1}{4}\Tr\Bigl(S_p \Sigma_{\tilde K}^{-1}\,
\nabla J_\infty(\tilde K)^\top\,\nabla J_\infty(\tilde K)\Sigma_{\tilde K}^{-1}S_p^{\top}\Bigr).
\end{align*}
Then by using 
$\|\Sigma_{\tilde K}^{-1}\|\le1/\underline\sigma(\Sigma_{\tilde K})$, we have 
\begin{align*}
\Tr \bigl(E_{\tilde K}^\top E_{\tilde K}\bigr)
\leq \frac{\|S_p\|^2}{4\underline\sigma^2(\Sigma_{\tilde K})}\|\nabla J_\infty(\tilde K)\|_F^2,    
\end{align*}
which combined with \eqref{eq:cost different - PL} yields exactly
\begin{align*}
J_\infty(\tilde K)-J_\infty(\tilde K_p^\star)
\leq 
\frac{\|\Sigma_{\tilde K_p^\star}\|\|S_p\|^2}
{4\underline\sigma(R)\underline\sigma^2(\Sigma_{\tilde K})}\|\nabla J_\infty(\tilde K)\|_F^2.
\end{align*}
\end{proof}
We now state that the gradient of the LQG cost with respect to $\tilde{K}$ satisfies the Lipschitz property. 
\begin{lemma}[Smoothness of the LQG Gradient]
\label{lemma:smoothness}
Given two stabilizing controllers $\tilde{K},\tilde{K}^\prime \in \tilde{\mathcal K}_{p,\alpha}$, it holds that 
\begin{align*}
\bigl\|\nabla J_\infty(\tilde K) - \nabla J_\infty(\tilde K^\prime)\bigr\|_F
\leq
L_g \,\bigl\|\tilde K - \tilde K^\prime \bigr\|_F,  
\end{align*}
where $L_g>0$ is the  Lipschitz constant. 
\end{lemma}

The proof of Lemma \ref{lemma:smoothness} follows directly from \cite[Lemma C.3]{gravell2020learning} under the change of variables $K = \tilde{K}S_p^\dagger$. Moreover, the Lipschitz constant $L_g$ is polynomial in the problem parameters, including $\|S_p^\dagger\|$, $\|A\|$, $\|B\|$, $\|\tilde{K}\|$, and $\|P_{\tilde{K}}\|$.

Lemmas~\ref{lemma:gradient dominance} and~\ref{lemma:smoothness} establish that the infinite-history LQG cost satisfies gradient dominance and smoothness on $\tilde{\mathcal K}_{p,\alpha}$ under the reachability condition. This is significant, as the standard state-space representation generally does not satisfy gradient dominance \cite{mohammadi_lack_2021}. Unfortunately, we often do not have access to the system and noise matrices and thus $S_p$ is not available. We elaborate on how to learn $S_p$ from data in Section \ref{sec: model-free}.

\textcolor{black}{\begin{remark}[Choosing the history length]
\label{rem:history_length}
Lemma~\ref{lem:cost-gradient-truncation} gives us a direct interpretation of the
history length $p$ as an approximation parameter. In particular, for any prescribed accuracies $\varepsilon_J,\varepsilon_\nabla>0$, it is
sufficient to choose $p$ sufficiently large according to
$$
p \ge \max\left\{\frac{\log(C_5/\varepsilon_J)}{\log(1/\rho_1)},\frac{\log(C_6/\varepsilon_\nabla)}{\log(1/\rho_1)}
\right\}
$$
to guarantee that 
$$
|J_p(K)-J_\infty(K)|\le \varepsilon_J \text{ and } \|\nabla J_p(K)-\nabla J_\infty(K)\|_F
\le \varepsilon_\nabla .
$$
Hence, the finite-history problem can approximate both the objective value and first-order geometry of the infinite-history problem to arbitrary accuracy, with a history length that grows 
logarithmically in the desired precision.
This approximation comes at a statistical and computational cost: increasing $p$ increases the history vector $z_{t,p}$, which increases the
dimension of the lifted controller, and requires a longer warm-up trajectory before the policy can be evaluated. Hence, $p$ controls a bias-complexity tradeoff between history truncation and the amount
of data required by the finite-history implementation.
\end{remark}}

For now, we consider the model-based setting (i.e., with $S_p$ known) and the update
\begin{equation}
\label{eq:gd_update}
\tilde{K}_{n+1} = \tilde{K}_{n} - \eta_n G_n,\;
G_n:=\nabla J_\infty\big(\tilde{K}_{n}\big)\Pi_p,
\end{equation}
where $\eta_n>0$ is the step-size and $\Pi_p=S_p^{\dagger}S_p$ guarantees  $\tilde K=\tilde K\Pi_p$.

\begin{assumption}\label{assumption:initial controller} The initial controller $\tilde{K}_0$ is stabilizing.
\end{assumption}

{\color{black}
\begin{assumption}
\label{ass:reachability}
The initial lifted controller $\tilde K_0$ is stabilizing and the pair $\left(A+B\tilde K_0S_p^{\dagger},\Sigma_\nu^{1/2}\right)$ is controllable.
\end{assumption}
This condition is needed because $\Sigma_\nu$ may be singular when $n_y<n_x$.

\begin{theorem}\label{thm:local-preservation}Suppose $\tilde K_n\in\tilde{\mathcal K}_p$ and $A_n:=A+B\tilde K_nS_p^{\dagger}$ is Schur, with $(A_n,\Sigma_\nu^{1/2})$ controllable. Let $G_n$ be any finite direction satisfying $G_n=G_n\Pi_p$. Then there exists $\bar\eta_n>0$ such that, for every $0\leq\eta_n\leq\bar\eta_n$, the update 
$$\tilde K_{n+1}=\tilde K_n-\eta_nG_n,$$
belongs to $\tilde{\mathcal K}_p$, $A_{n+1}$ is Schur, and $(A_{n+1},\Sigma_\nu^{1/2})$ is controllable.
\end{theorem}
\begin{proof}
We first note that $G_n=G_n\Pi_p$ preserves $\tilde K_n=\tilde K_n\Pi_p$. Moreover, $\rho(A+B\tilde KS_p^{\dagger})$ is continuous in $\tilde K$. As $\rho(A_n)<1$, stability holds in a neighborhood of $\tilde K_n$. We define the controllability matrix
\begin{align*}
\mathcal R(\tilde K):=[\Sigma_\nu^{1/2}\ A_{\tilde K}\Sigma_\nu^{1/2}\ \cdots\ A_{\tilde K}^{n_x-1}\Sigma_\nu^{1/2}].
\end{align*}
The map $\tilde K\mapsto\sigma_{\min}(\mathcal R(\tilde K))$ is also continuous and is positive at $\tilde K_n$. Hence reachability holds in a possibly smaller neighborhood. Then, choosing $\bar\eta_n$ so that $\eta_n\|G_n\|_F$ lies within both neighborhoods completes the proof.
\end{proof}

By starting from Assumption~\ref{ass:reachability}, Theorem~\ref{thm:local-preservation} applies inductively. For any finite number of iterations $N$, the positive quantity $\sigma_0:=\min_{0\leq n\leq N}\underline\sigma(\Sigma_{\tilde K_n})$ therefore exists.}

Assumption~\ref{assumption:initial controller} is standard in policy gradient methods for the LQR problem \cite{fazel_global_2019, mohammadi2020linear, gravell2020learning, toso2024oracle}, and the same requirement naturally extends to the LQG setting. In particular, if the initial policy $u_t = \tilde{K}_0 z_{t,p}$ does not stabilize the system \eqref{eq:sys}, then the policy gradient $\nabla J_\infty(\tilde{K}_n)$ is not well-defined. In such a case, no positive step size $\eta$ ensures that the updated policy $\tilde{K}_{n+1}$ remains within the stabilizing set $\tilde{\mathcal K}_{p,\alpha}$. We also emphasize a recent line of work \cite{lamperski2020computing, perdomo2021stabilizing, zhao2024convergence, toso2025learning} that addresses the problem of computing an initial stabilizing controller via a sequence of discounted LQR problems. In Section~\ref{sec:numerics}, we demonstrate that a similar approach can be adapted to the LQG setting for learning $\tilde{K}_0$ from trajectory data.

\begin{lemma}
\label{lemma:convergence}
Given $\tilde{K}_0 \in \tilde{\mathcal K}_p$, suppose Assumptions~\ref{ass:history} and~\ref{ass:reachability} hold, and select $0<\eta_n\leq\min\{1/L_g,\bar\eta_n\}$ as in Theorem~\ref{thm:local-preservation}. Then, the controller update \eqref{eq:gd_update} satisfies the following contraction for the infinite-history cost $J_\infty$:
\begin{align*}
J_\infty(\tilde{K}_{n+1}) - J_\infty(\tilde{K}_p^\star)\hspace{-0.075cm} \leq \rho_{\mathrm{c}}(\eta_n)\hspace{-0.05cm}\left(\hspace{-0.05cm} J_\infty(\tilde{K}_n) \hspace{-0.05cm}- \hspace{-0.05cm}J_\infty(\tilde{K}_p^\star)\hspace{-0.05cm}\right)\hspace{-0.075cm}.
\end{align*}
where $\rho_{\mathrm{c}}(\eta_n):=1- ({2 \eta_n \underline{\sigma}^2(\Sigma_{\tilde{K}_n}) \underline{\sigma}(R)})/({\|\Sigma_{\tilde K_p^\star}\| \|S_p\|^2})$. 

\begin{proof}
The proof is routine and follows directly by combining Lemmas~\ref{lemma:smoothness} and \ref{lemma:gradient dominance}. We provide a detailed derivation of the analogous result for model-free in Section \ref{sec: model-free}.
\end{proof}
\end{lemma}

Therefore, if the infinite-history iterate satisfies $J_\infty(\tilde K_N)-J_\infty(\tilde K_p^\star)\leq\epsilon$, Lemma~\ref{lem:cost-gradient-truncation} yields
\begin{align*}
J_p(\tilde K_N)-J_p(\tilde K_p^\star)\leq\epsilon+2C_5\rho_1^p.
\end{align*}
\noindent \textbf{Discussion:} We note that under Assumptions~\ref{ass:history} and~\ref{ass:reachability}, Lemma~\ref{lemma:convergence} establishes that the update rule in \eqref{eq:gd_update} converges to the optimal lifted controller $\tilde{K}_p^\star$ for the infinite-history cost, while Theorem~\ref{thm:local-preservation} preserves stability and reachability at every iteration. Given access to the lifting matrix $S_p$, the optimal LQG controller can be recovered as $K^\star = \tilde{K}_p^\star S_p^{\dagger}$. If $\eta_n\geq\underline{\eta}>0$, the infinite-history gap is at most $\epsilon$ after
\begin{align*}N \geq \frac{\|\Sigma_{\tilde K_p^\star}\| \|S_p\|^2}{2\underline{\eta}\underline{\sigma}(R)\sigma_0^2}\log\left(\frac{J_\infty(\tilde{K}_0) - J_\infty(\tilde{K}_p^\star)}{\epsilon}\right),\end{align*}
iterations. This iteration bound is conditional on the selected step sizes admitting the uniform lower bound $\underline\eta>0$; Theorem~\ref{thm:local-preservation} guarantees a positive admissible step size at each iterate but does not, by itself, provide a uniform lower bound over an infinite sequence.

Figure~\ref{fig:differentP} illustrates the empirical effect of the history length $p$ on PG-LQG. The theoretical iteration bound depends on $\|S_p\|^2$, but it does not imply that this norm decreases with $p$: in fact, $S_pS_p^\top$ is nondecreasing in the positive-semidefinite order. The faster convergence observed for this example is therefore an empirical property of the resulting optimization trajectories, not a general monotonic consequence of increasing $p$.

\begin{figure}[t]
    \centering
    \includegraphics[width=1\columnwidth]{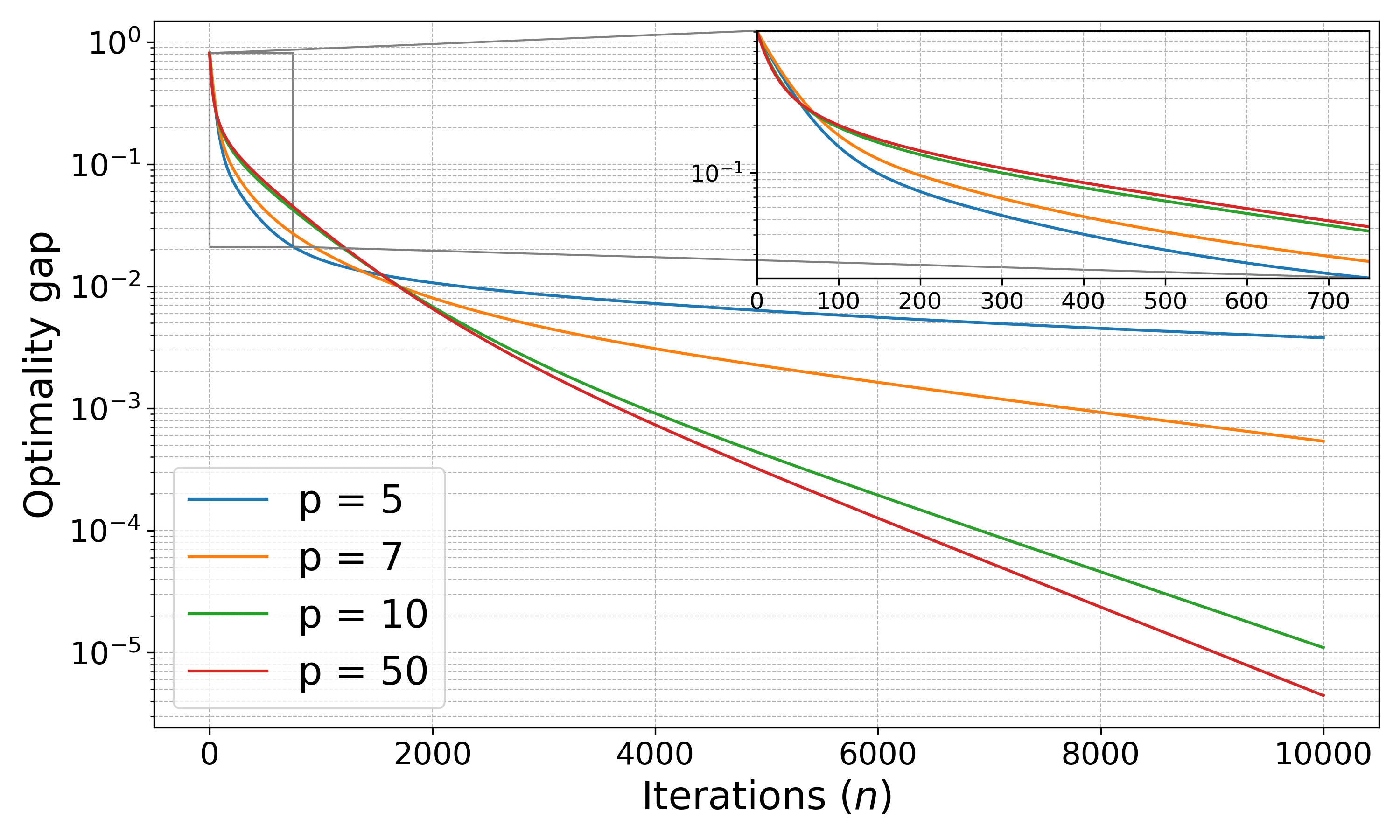}
    \caption{Model-based optimality gap for varying history length $p$.\hspace{0.2cm}}
    \label{fig:differentP}
\end{figure}

\section{Model-free Analysis} \label{sec: model-free}

We are now ready to address the setting where the system and noise matrices $(A,B,C, V, W)$ are unavailable. We assume access to the exact finite-history cost oracle $J_p$ and to the row-space projector $\Pi_p:=S_p^\dagger S_p$. The latter is side information required to enforce the policy class in Assumption~\ref{ass:history}. We adopt the standard zeroth-order estimator of $\nabla J_p(\tilde K)$ \cite{malik_derivative-free_2020,nesterov2017random}. Let $d_p:=n_u\operatorname{rank}(\Pi_p)=n_un_x$ and impose $U_i=U_i\Pi_p$. Given $\tilde K$, let $U_i$ be drawn uniformly from the corresponding $d_p$-dimensional sphere with $\|U_i\|_F=r$. The estimator is
\begin{align*}
\texttt{ZO}(\tilde{K},n_s,r)&:= \widehat{\nabla}J_p(\tilde{K}) \\
&:= \frac{d_p}{n_s}\sum_{i=1}^{n_s}\frac{J_p(\tilde{K} + U_i)\,U_i}{r^2}.
\end{align*}
where $n_s$ is the number of oracle calls.

The history is initialized exactly as specified in Section~\ref{sec:LQG}: apply the fixed inputs $u_0,\ldots,u_{p-1}$, record $y_0,\ldots,y_p$, construct $z_{p,p}$, and start the policy at time $p$. For a stabilizing policy, a zero warm-up input is admissible and its transient contribution vanishes from the ergodic cost.

Once the warm-up phase is complete, the history vector $z_{t,p}$ is available for $t\geq p$. The cost oracle evaluates each perturbed policy $u_t=(\tilde K+U_i)z_{t,p}$, and the resulting zeroth-order estimate is used in Algorithm~\ref{alg:model-free}. The projector $\Pi_p$ enforces the policy class required by the analysis and is assumed known.

\begin{algorithm}
\caption{Model-Free PG-LQG}
\label{alg:model-free}
\KwIn{$\tilde{K}_{0} \text{ (see Section \ref{sec:numerics})}$, $\Pi_p$, $r$, $n_s$, and  $N$}
\For{$n = 0$ \KwTo $N-1$}{
    Estimate gradient: $\widehat{\nabla} J_p(\tilde{K}_{n}) \gets \texttt{ZO}(\tilde{K}_{n}, n_s, r)$ 
  
    Select $\eta_n\leq\bar\eta_n$ as in Theorem~\ref{thm:local-preservation}

    Policy gradient update:   $\tilde{K}_{n+1} = \tilde{K}_{n} - \eta_n\, \widehat{\nabla} J_p(\tilde{K}_{n})$ 
    
}
\Return $\tilde{K}_N$
\end{algorithm}

\vspace{0.2cm}
\noindent \textbf{Zeroth-order estimation error:} Define the smoothed finite-history cost by
\begin{align*}
\nabla J_{p,r}(\tilde K):=\frac{d_p}{r^2}\mathbb E[J_p(\tilde K+U)U].
\end{align*}
The error with respect to the infinite-history gradient decomposes as
\begin{align*}
&\|\widehat{\nabla}J_p(\tilde K)-\nabla J_\infty(\tilde K)\|_F \leq \|\widehat{\nabla}J_p(\tilde K)-\nabla J_{p,r}(\tilde K)\|_F \\
&+\|\nabla J_{p,r}(\tilde K)-\nabla J_p(\tilde K)\|_F +\|\nabla J_p(\tilde K)-\nabla J_\infty(\tilde K)\|_F.
\end{align*}

The three terms are the sampling error, smoothing bias, and history-truncation error, respectively. Smoothness of $J_\infty$ and Lemma~\ref{lem:cost-gradient-truncation} give
\begin{align*}
\|\nabla J_{p,r}(\tilde K)-\nabla J_p(\tilde K)\|_F
&\leq L_gr+2C_6\rho_1^p.
\end{align*}
Indeed, insert the infinite-history gradient at the perturbed and unperturbed policies, apply Lemma~\ref{lemma:smoothness} to the middle difference, and use Lemma~\ref{lem:cost-gradient-truncation} for the two remaining differences. Thus, the truncation bound transfers infinite-history smoothness to the finite-history smoothing bias with an additive error, rather than implying that $\nabla J_p$ is itself $L_g$-Lipschitz.

To control the variance term, matrix concentration inequalities \cite{vershynin2018high} give, at each iteration, with probability at least $1-\delta/N$,
\begin{align*}
\left\| \widehat{\nabla} J_p(\tilde{K}) - \nabla J_{p,r}(\tilde{K}) \right\|_F \leq \frac{C_7\sigma_r \log(N/\delta)}{\sqrt{n_s}},
\end{align*}
for $C_7=d_pJ_\textit{max}$, where $\sigma_r=\mathcal O(1/r)$ and $J_\textit{max}=\max_{\tilde K\in\tilde{\mathcal K}_{p,\tilde K^\prime,\alpha}}J_p(\tilde K)$. A union bound and Lemma~\ref{lem:cost-gradient-truncation} therefore give, simultaneously for all $N$ iterations with probability at least $1-\delta$,
\begin{align}\label{eq:gradient estimation error}
 \left\| \widehat{\nabla} J_p(\tilde K) - \nabla J_\infty(\tilde K) \right\|_F \leq L_gr + \frac{C_7\sigma_r \log(N/\delta)}{\sqrt{n_s}} +3C_6\rho_1^p.  
\end{align}

Hence, by appropriately tuning the smoothing radius $r$ and the number of oracle calls $n_s$, the overall estimation error can be controlled within a desired precision. We analyze the one-point estimator, whose variance scales as $\mathcal O(1/r^2)$. 

\begin{theorem}[Convergence]
\label{thm:K convergence}
Given an initial stabilizing controller $\tilde{K}_0 \in \tilde{\mathcal K}_p$, a small probability $\delta \in (0,1)$, and accuracy $\epsilon>0$. Suppose Assumptions~\ref{ass:history} and~\ref{ass:reachability} hold, the projector $\Pi_p$ and exact cost oracle $J_p$ are available, and $3C_6\rho_1^p\leq\sqrt{\epsilon/(24\mu_{\texttt{PL}})}$. Suppose the number of samples $n_s$, smoothing radius $r$, step sizes $\eta_n$, and number of iterations $N$ satisfy:
\begin{align*}
n_s &\geq \frac{24 C_7^2 \mu_{\texttt{PL}}\sigma^2_r \log^2(N/\delta)}{\epsilon},\\
r &\leq \sqrt{\frac{\epsilon}{24\mu_{\texttt{PL}}L^2_g}},
\end{align*}
\begin{align*}
0<\underline{\eta}\leq\eta_n&\leq\min\left\{\frac{1}{4L_g}, \frac{\mu_\text{PL}}{2},\bar\eta_n\right\},\\
N&\geq  \frac{4\mu_{\texttt{PL}}}{\underline{\eta}} \log\left(\frac{2\Delta J_\infty^\star(\tilde{K}_0)}{\epsilon}\right).
\end{align*}
where $\Delta J_\infty^\star(\tilde{K}_0) := J_\infty(\tilde{K}_0) - J_\infty(\tilde{K}_p^\star)$ is the initial infinite-history optimality gap. Then Algorithm \ref{alg:model-free} produces $\tilde{K}_N$ that satisfies
\begin{align*}
J_p(\tilde{K}_N) - J_p(\tilde{K}_p^\star) \leq \epsilon+2C_5\rho_1^p.
\end{align*}
with probability $1-\delta$.
\end{theorem}

\begin{proof}
We begin the proof with the Lipschitz property of the gradient (Lemma \ref{lemma:smoothness}):
\begin{align*}
\Delta J_\infty(\tilde{K}_n) \leq  \left\langle \nabla J_\infty(\tilde{K}_{n}), \Delta \tilde{K}_n \right\rangle + \frac{L_g}{2} 
\| \Delta \tilde{K}_n \|_F^2,
\end{align*}
where $\Delta J_\infty(\tilde{K}_n) := J_\infty(\tilde{K}_{n+1}) - J_\infty(\tilde{K}_{n})$ denotes the one-step infinite-history cost difference and $\Delta \tilde{K}_n := \tilde{K}_{n+1} - \tilde{K}_{n}$. Therefore,
\begin{align}\label{eq:optimality gap iteration n}
\Delta J_\infty(\tilde{K}_n) 
&\leq -\eta_n \left\langle \nabla J_\infty(\tilde{K}_{n}), \widehat{\nabla} J_p(\tilde{K}_{n}) \right\rangle\notag\\
&+ \frac{L_g\eta_n^2}{2}
\| \widehat{\nabla} J_p(\tilde{K}_{n}) \|_F^2.
\end{align}

In addition, we control the norm of the estimated gradient using Young’s inequality:
\begin{align*}
\|\widehat{\nabla} J_p(\tilde{K}_{n})\|_F^2 
&\leq 2 \| \nabla J_\infty(\tilde{K}_{n}) \|_F^2\\
&+2 \| \widehat{\nabla} J_p(\tilde{K}_{n}) \hspace{-0.05cm}-\hspace{-0.05cm} \nabla J_\infty(\tilde{K}_{n}) \|_F^2,
\end{align*}
and by applying the above to \eqref{eq:optimality gap iteration n} we obtain
\begin{align*}
\Delta J_\infty(\tilde{K}_n)
&\leq -\frac{\eta_n}{4} \| \nabla J_\infty(\tilde{K}_n) \|_F^2 \\
&+ \frac{3\eta_n}{4} \| \widehat{\nabla} J_p(\tilde{K}_n) - \nabla J_\infty(\tilde{K}_n) \|_F^2,
\end{align*}
by selecting the step-size as $\eta_n \leq \frac{1}{4L_g}$. If the smoothing radius $r$ and number of samples $n_s$ satisfy
\begin{align}\label{eq:smoothing radius and number of samples}
r \leq \frac{\sqrt{\epsilon_\texttt{est}}}{2 L_g},  \text{ and } n_s \geq \frac{4 C_7^2\sigma^2_r \log^2(N/\delta)}{\epsilon_{\texttt{est}}},
\end{align}
then from \eqref{eq:gradient estimation error} it follows that with probability $1-\delta$,
\begin{align*}
\Delta J_\infty(\tilde{K}_n)
&\leq -\frac{\eta_n}{4} \| \nabla J_\infty(\tilde{K}_n) \|_F^2  + \frac{3\eta_n}{4}\epsilon_{\texttt{est}}
\end{align*}
for some sufficiently small estimation error $\epsilon_{\texttt{est}}$. We now leverage the gradient dominance property of the cost under the history representation (Lemma \ref{lemma:gradient dominance}) to write 
\begin{align}\label{eq:probability step equation}
\Delta J_\infty^\star(\tilde{K}_{n+1})
&\leq \left(1-\frac{\eta_n}{4 \mu_{\texttt{PL}}}\right) \Delta J_\infty^\star(\tilde{K}_{n}) 
+ \frac{3\eta_n}{4}\epsilon_{\texttt{est}},
\end{align}
where $\Delta J_\infty^\star(\tilde{K}) := J_\infty(\tilde{K}) - J_\infty(\tilde{K}_p^\star)$ denotes the infinite-history optimality gap. Unrolling the above expression over $N$ iterations gives
\begin{align*}
\Delta J_\infty^\star(\tilde{K}_{N})
&\leq \left(1-\frac{\underline{\eta}}{4 \mu_{\texttt{PL}}}\right)^N \Delta J_\infty^\star(\tilde{K}_{0}) 
+ 3\mu_{\texttt{PL}}\epsilon_{\texttt{est}},
\end{align*}
where we set the number of iterations $N$ according to 
\begin{align*}
N \geq \frac{4\mu_{\texttt{PL}}}{\underline{\eta}} \log\left(\frac{2\Delta J_\infty^\star(\tilde{K}_0)}{\epsilon}\right),
\end{align*}
with step sizes $\eta_n \leq  \frac{\mu_{\texttt{PL}}}{2}$ to obtain $J_\infty(\tilde{K}_N) - J_\infty(\tilde{K}_p^\star) \leq \epsilon$, and hence $J_p(\tilde{K}_N)-J_p(\tilde K_p^\star)\leq\epsilon+2C_5\rho_1^p$,
for some estimation error $\epsilon_{\texttt{est}} \leq \frac{\epsilon}{6\mu_{\texttt{PL}}}$.
\end{proof}
\vspace{0.2cm}

\noindent\textbf{Discussion:} Conditional on Assumptions~\ref{ass:history} and~\ref{ass:reachability}, Theorem~\ref{thm:K convergence} returns a controller whose finite-history cost is within $\epsilon+2C_5\rho_1^p$ of the benchmark $\tilde K_p^\star$. The first term is the optimization error and the second is the history-truncation error.

{\color{black}
Theorem~\ref{thm:K convergence} characterizes how the different approximations enter the model-free guarantee. The optimization accuracy
is controlled through the number of policy-gradient iterations, while
the zeroth-order gradient estimate is determined jointly by the smoothing radius $r$ and the number of cost-oracle calls $n_s$. On the other hand, the term $2C_5\rho_1^p$ is intrinsic to the
finite-history representation and cannot be reduced by running
additional policy-gradient iterations or making more oracle calls at a
fixed history length. Therefore, achieving a given accuracy
requires balancing these three sources of error: optimization error,
zeroth-order estimation error, and history truncation.

The dependence on $r$ and $n_s$ also highlights the standard bias-variance
tradeoff of zeroth-order estimation. By decreasing $r$, we reduce the
smoothing bias, but increases the variance of the one-point estimator
through $\sigma_r=\mathcal{O}(1/r)$, and thus requires a larger number of
oracle calls to keep the same estimation accuracy. On the
other hand, by increasing the history length $p$, we can reduce the finite-history representation
error geometrically, so the truncation contribution can be made small
with a history length that grows only logarithmically with the given accuracy. Hence, the final accuracy is governed by two resources: additional oracle calls are required to
improve gradient estimation, whereas a longer history is required to
reduce the approximation error induced by truncating the input-output
history.
}

\begin{remark}[\textbf{Recovering $S_p$ and $K^\star$}] \label{remark:recovering S} Theorem~\ref{thm:K convergence} applies to the set $\tilde{\mathcal K}_p$ in Assumption~\ref{ass:history} and assumes that the row-space projector $\Pi_p$ is available as side information. History-state regression can estimate this representation when state measurements are available~\cite{guo_imitation_2023}; it does not provide an output-only construction. Jointly estimating $\Pi_p$ from input-output data and updating the policy is left for future work.
\end{remark}

\begin{remark}[\textbf{History length $p$ tradeoff}] Increasing $p$ reduces the history-truncation error geometrically, but it does not generally improve the optimization rate monotonically. A larger $p$ also increases the dimension of the lifted policy and the length of the warm-up control sequence. This introduces a tradeoff between approximation accuracy, computational complexity, and the required history length. In data-limited regimes, it may be preferable to use a shorter history, provided that $p\geq\nu_r$ so that $S_p$ has full row rank and the small-gain condition~\eqref{eq:history-small-gain} holds.
\end{remark}

The convergence guarantees are conditional on Assumptions~\ref{ass:history} and~\ref{ass:reachability}. We do not claim that an unconstrained update over arbitrary history gains preserves stability: such gains need not admit an equivalent state-feedback realization. Therefore, under the restriction $\tilde K\in\tilde{\mathcal K}_p$, the problem inherits the favorable LQR geometry, while Lemma~\ref{lem:truncation} quantifies the finite-history approximation.

\section{Numerical Validation}\label{sec:numerics}

We now validate the convergence of our proposed approach%
\footnote{Code to reproduce the results can be found at \url{https://github.com/jd-anderson/PG-LQG}.}
and introduce a heuristic for finding an initially stabilizing
controller. We assess the performance of our model-free PG-LQG
algorithm on an unstable system with four states, two inputs, and
two outputs.

{\color{black}

The system matrices used throughout the numerical experiments are

$$
\setlength{\arraycolsep}{4pt}
A =
\begin{bmatrix}
-0.2639 &  0.5924 & -0.6445 & -0.8047 \\
 0.5288 &  0.4654 &  0.6087 &  0.0537 \\
-0.2803 & -0.4883 &  0.1135 & -0.6962 \\
-1.0480 &  0.2543 & -0.1278 & -0.1279
\end{bmatrix}.
$$

$$
\setlength{\arraycolsep}{4pt}
B =
\begin{bmatrix}
-0.9313 & -1.0678 \\
 2.0774 &  \phantom{-}0.3084 \\
-1.4758 & -0.7451 \\
-0.2621 & -1.5536
\end{bmatrix}.
$$

$$
\setlength{\arraycolsep}{4pt}
C =
\begin{bmatrix}
 2.2795 & -0.6637 & -1.1390 & -0.8495 \\
 0.4608 &  1.2424 &  1.4244 & -1.3973
\end{bmatrix}.
$$

The eigenvalues of the open-loop state matrix are
$\operatorname{spec}(A)=\{-1.5,\;0.9415,\;0.3728\pm0.6378\,\mathrm{i}\}$,
and hence, with $\rho(A)=1.5$, the system is open-loop unstable.

}

The process- and measurement-noise covariances are given by
$W=0.01\,I_{4}$ and $V=0.01\,I_{2}$, respectively, and the
cost matrices are set to identity, i.e.,
$Q=I_{2}$ and $R=I_{2}$. We use a rollout horizon of
$T=100$, smoothing radius $r=0.1$, and number of samples
$n_s=1000$ throughout the experiments.

As the system above is unstable, in order to collect useful information in the warm-up phase, we present an approach to learn an initial stabilizing controller $\tilde{K}_0$ by solving a sequence of discounted LQG problems under the proposed history representation, similar to the discount annealing method for LQR \cite{perdomo2021stabilizing, lamperski2020computing, zhao2024convergence, toso2025learning}. We emphasize that developing theoretical guarantees for our adaptation to the LQG setting is beyond the scope of this work and that no attempt to optimize its performance has been made.

To set the stage, we begin by defining the discounted LQG problem. Given a discount factor $\gamma \in (0,1]$, the objective is to minimize the following discounted cost:
\begin{align*}
    J_\gamma :=\mathbb{E} \left[ \sum_{t=0}^{\infty} \gamma^t\left( y_t^\top Q^\prime y_t + u_t^\top R^\prime u_t \right) \right], \text{ s.t. \eqref{LQG}}.
\end{align*}
which is equivalent to minimizing
\begin{align*}
    J_\gamma :=\mathbb{E} \left[ \sum_{t=0}^{\infty} \left( \tilde{y}_t^\top Q^\prime \tilde{y}_t + \tilde{u}_t^\top R^\prime \tilde{u}_t \right) \right],
\end{align*}
subject to the damped system dynamics:
\begin{align*}
    \tilde{x}_{t+1} &= A_\gamma \tilde{x}_t + B_\gamma \tilde{u}_t + \gamma^{\frac{t+1}{2}}w_t, \notag \\
    \tilde{y}_t &= C \tilde{x}_t + \gamma^{\frac{t}{2}}v_t, \text{ for } t = 0,1,2,\ldots,
\end{align*}
where $A_\gamma = \sqrt{\gamma} A$ and $B_\gamma = \sqrt{\gamma}B$ are the damped system matrices, and $\tilde{x}_t = \gamma^{t/2}x_t$, $\tilde{u}_t = \gamma^{t/2}u_t$, and $\tilde{y}_t = \gamma^{t/2}y_t$ denote the state, input, and output of the damped system, respectively. Note that $Q^\prime$ and $R^\prime$ are simply artifacts in this approach, as we do not care about optimality in this stage. 

Under the proposed history representation, we aim to minimize $J_\gamma(\tilde{K})$ over the set of stabilizing controllers for the corresponding damped system:
\begin{align*}
\tilde{\mathcal K}_{p,\gamma} := \left\{ \tilde{K} \,\middle|\, \sqrt{\gamma}\,\rho\left(A + B \tilde{K} S_p^{\dagger}\right) < 1 \right\}.
\end{align*}
When the discount factor $\gamma$ is sufficiently small, specifically $\gamma \leq 1/\rho(A)^2$, then $\tilde{K} = 0$ stabilizes the  damped system. This motivates our approach to learn an initial stabilizing controller $\tilde{K}_0 \in \tilde{\mathcal K}_p$ for Algorithm~\ref{alg:model-free} by first selecting a small  discount factor $\gamma$ such that $\tilde{K} = 0$ is stabilizing. With this initial controller, we employ PG to solve $J_\gamma$, yielding the controller $K^\gamma_N$, which denotes the result of running PG-LQG for $N$ iterations at the current discount factor $\gamma$. 

To gradually approach the original (undiscounted) LQG problem, we iteratively increase $\gamma$, ensuring at each step that  $K^\gamma_N$ remains stabilizing for the updated damped system. This is achieved through a binary search procedure, as in~\cite{perdomo2021stabilizing}, to find a new discount factor $\gamma_+ \in [\gamma, 1]$. For this purpose, we select $\gamma_+$ such that $C_8 J_\gamma(K^\gamma_N) \leq J_{\gamma_+}(K^\gamma_N) \leq  C_9 J_\gamma(K^\gamma_N),$
for some positive constants $C_8$ and $C_9$. That is, the updated discount factor $\gamma_+$ is selected to ensure that $K^\gamma_N$ remains stabilizing for the PG-LQG problem under the new discount factor $\gamma_+$. 

{\color{black}

The resulting discount-annealing approach is summarized in
Algorithm~\ref{alg:discount-annealing}. We start from a sufficiently
discounted problem for which an initial stabilizing controller is
available (e.g., the trivial zero control action). We then solve a sequence of discounted PG-LQG problems while progressively increasing the discount factor toward one. At each step, the controller obtained for the current discounted problem is used to
initialize the subsequent one.

{\color{red}

\IncMargin{1.0em}
\begin{algorithm}[t]
\caption{Discount-Annealing Initialization for PG-LQG}
\label{alg:discount-annealing}

\KwIn{Initial discount factor $\gamma_0\in(0,1)$,
inner-loop iterations $N$, and constants $0<C_8\leq 1 < C_9$}

\KwOut{Stabilizing initialization
$\widetilde K_0\in\widetilde{\mathcal K}_p$}

$\gamma\gets\gamma_0$ and
$\widetilde K^{\mathrm{cur}}\gets0$

\While{$\gamma<1$}{

Starting from $\widetilde K^{\mathrm{cur}}$, run PG-LQG
for $N$ iterations on $J_\gamma$

Denote the resulting controller by
$\widetilde K^+$

Use binary search to select $\gamma_+\in[\gamma,1]$ such that
\begin{align*}
C_8J_\gamma(\widetilde K^+)
\leq
J_{\gamma_+}(\widetilde K^+)
\leq
C_9J_\gamma(\widetilde K^+).
\end{align*}

\If{$\gamma_+=1$}{
    $\widetilde K_0\gets\widetilde K^+$
    
    \Return $\widetilde K_0$
}

$\widetilde K^{\mathrm{cur}}\gets\widetilde K^+$ and
$\gamma\gets\gamma_+$
}

\end{algorithm}
\DecMargin{1.0em}

}

}Figure~\ref{fig:learningK0} illustrates the evolution of the spectral radius of the closed-loop system, i.e., $\rho(A + B \tilde{K} S_p^{\dagger})$, and discount factor $\gamma$, over the iterations of the discount annealing procedure for the LQG problem under the proposed history representation and experimental setup. As shown, after a finite number of discount factor updates, the method successfully learns an initial stabilizing controller $\tilde{K}_0 \in \tilde{\mathcal K}_p$ directly from trajectory data. This controller is then used to initialize Algorithm~\ref{alg:model-free}, which proceeds to learn a near-optimal LQG controller.

\begin{figure}[t]
    \centering
    \includegraphics[width=1\columnwidth]{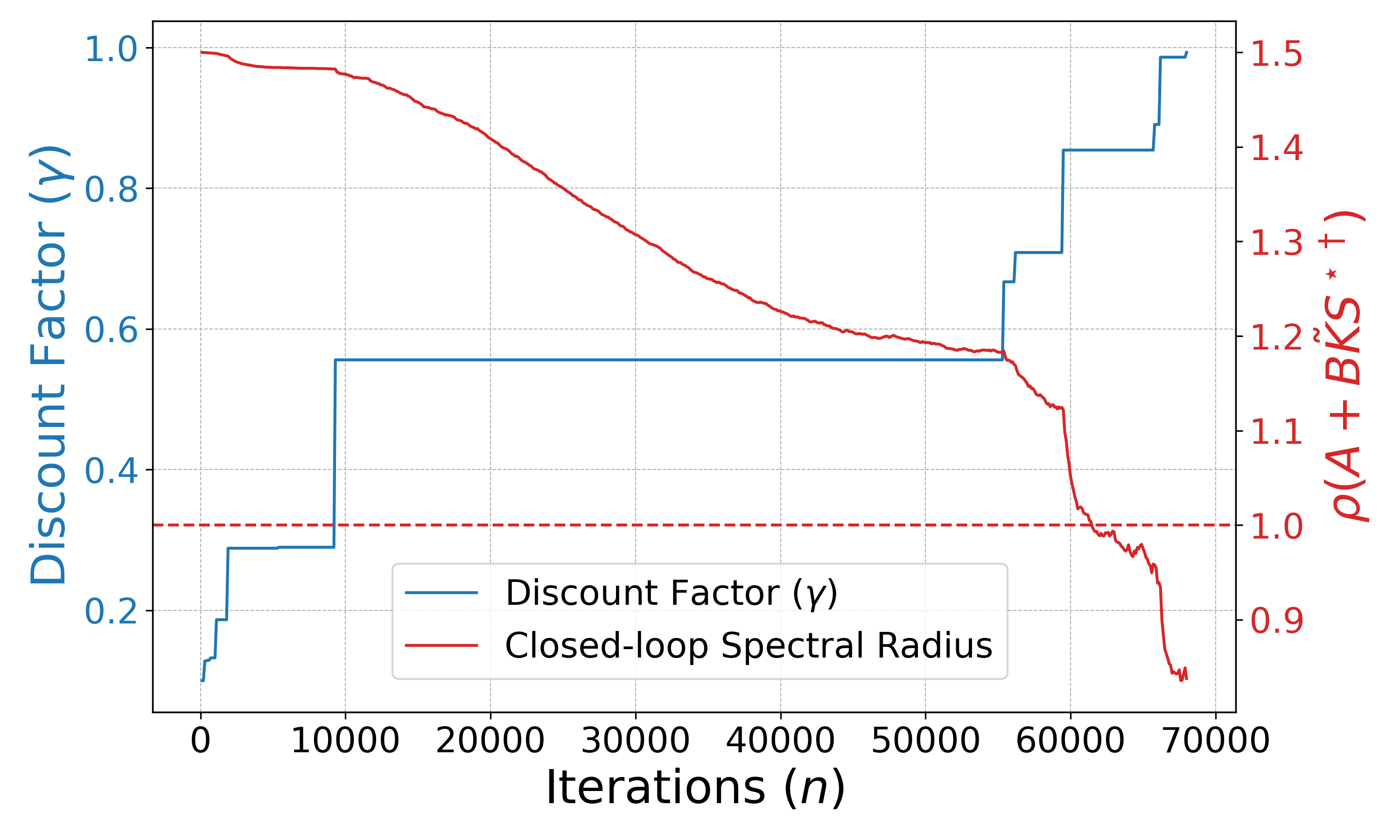}
    \caption{Evolution of the discount factor and spectral radius
    for learning $\tilde{K}_0$.}
    \label{fig:learningK0}
\end{figure}
Next, we show the convergence of Algorithm~\ref{alg:model-free}.

Figure~\ref{fig:modelfree} illustrates the convergence of Algorithm~\ref{alg:model-free} in the model-free setting. As observed, the presence of measurement and process noise, and gradient estimation errors introduce a bias, resulting in a non-zero optimality gap at convergence. Nevertheless, the algorithm consistently converges toward a near-optimal solution. This result confirms our trends on the convergence in the model-free setting.
\begin{figure}[t]
    \centering
    \includegraphics[width=1\columnwidth]{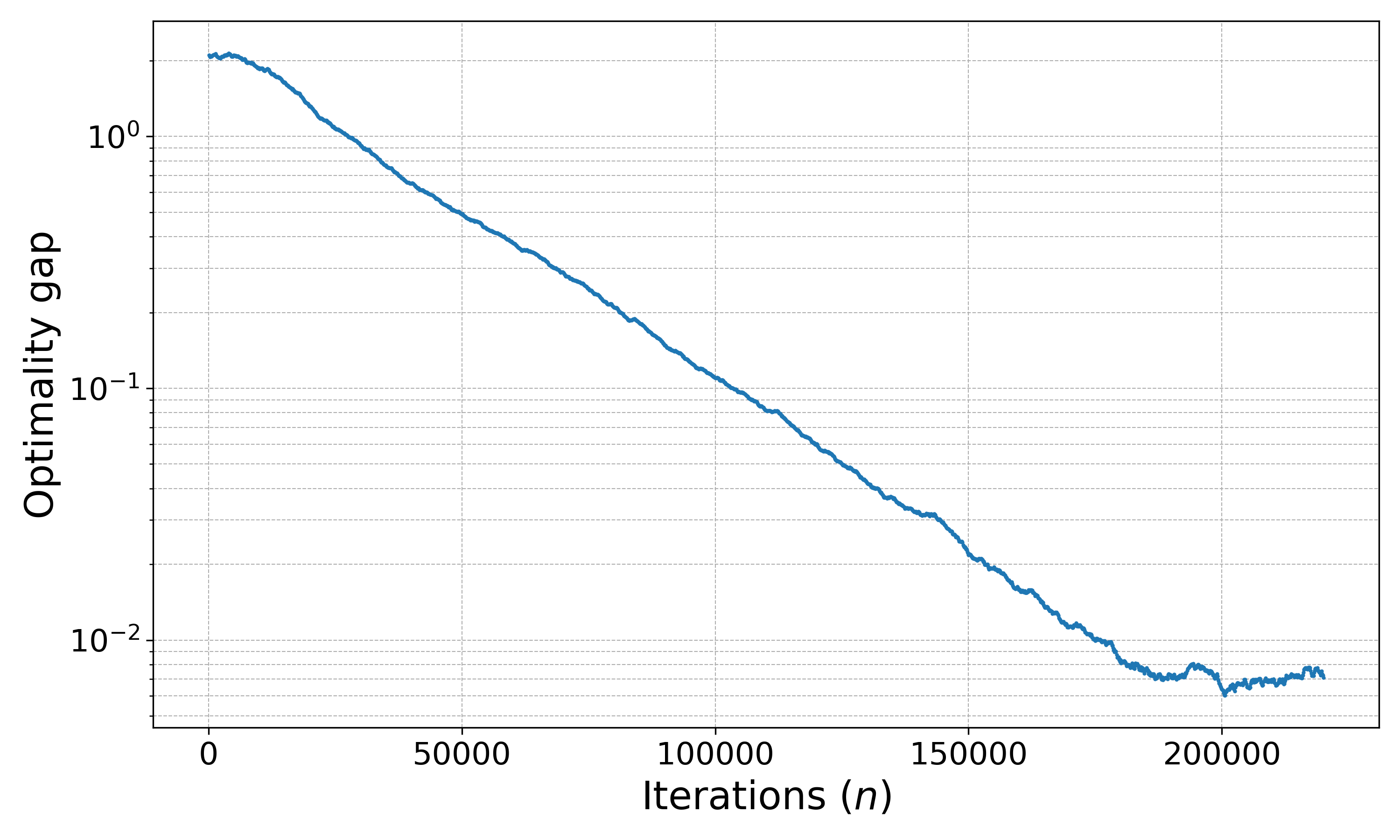}
    \caption{Optimality gap of Algorithm~\ref{alg:model-free}
    over the number of iterations.}
    \label{fig:modelfree}
\end{figure}
{\color{black}
We note that the nonzero terminal gap in Fig.~\ref{fig:modelfree} is
 consistent with our result in Theorem~\ref{thm:K convergence}. That is,  even after the optimization error has been substantially reduced over iterations, finite-history
truncation and gradient-estimation error induce a residual accuracy
error. The experiment illustrates the difference between
algorithmic convergence and exact recovery of the optimal controller.
}
Beyond confirming the convergence, these results also highlight different important tradeoffs. First, the choice of history length $p$ critically impacts the convergence rate, as longer histories lead to better conditioning of the optimization landscape but also increase the warm-up horizon. Second, while the policy gradient method achieves near-optimal performance, noise and gradient estimation error prevent exact recovery of the optimal controller, even with large sample sizes, motivating future work on variance reduction techniques as in \cite{toso2024oracle}.

\section{Conclusions and Future Work}

{\color{black}

We studied policy gradient for the LQG problem under a history-based parameterization of the controller.
By restricting the lifted controller to
$\tilde K\in\tilde{\mathcal K}_p$, we demonstrated that the
infinite-history LQG objective recovers the favorable optimization
geometry of state-feedback LQR, namely gradient dominance and
smoothness, under a closed-loop reachability condition. For a finite
history length $p$, we quantified the underlying approximation error and
demonstrated that both the cost and policy gradient converge geometrically to their infinite-history counterparts. These properties yield convergence
guarantees for PG-LQG in both the model-based and model-free settings,
with the latter relying on a one-point zeroth-order gradient estimator
and converging up to the finite-history truncation error. We further
showed that a discount-annealing approach can be used to obtain an initial stabilizing lifted controller directly from trajectory data.

Many directions are still open. A fully data-driven implementation would
benefit from jointly learning the history representation and optimizing
the controller, rather than assuming access to the corresponding
row-space projector. Improving the statistical efficiency of the
model-free method, for example through variance-reduced gradient
estimation \cite{toso2024oracle}, is another important direction. The present framework also
naturally motivates extensions beyond the single-task setting.
In addition, subsequent work has extended these ideas to multitask LQG control and established corresponding performance and generalization
guarantees~\cite{toso2026multitask}. Providing analogous guarantees
under adversarial learning
is an interesting direction, particularly in light of recent results for robust multitask linear-quadratic
control~\cite{fallah2025adversarially}.

}

\section*{Acknowledgments}

The authors thank the anonymous reviewers for their valuable comments and constructive feedback, which helped improve the manuscript. Leonardo F. Toso is funded by the Center for AI and Responsible Financial Innovation (CAIRFI) Fellowship and by the Columbia Presidential Fellowship. James Anderson is partially funded by NSF grants EECS 2144634 and CNS 2535097 and the Center of AI Technology (CAIT) in collaboration with Amazon.

\bibliographystyle{IEEEtran}
\bibliography{references}

\end{document}